\documentclass[smallextended]{svjour3}       
\smartqed  
\usepackage{graphicx}
\usepackage{amsmath,amssymb,amsfonts,graphicx,latexsym,a4wide, enumerate, xcolor}
\usepackage{srcltx}
\usepackage{multirow}
\usepackage{mathtools}
\usepackage{tabularx}
\usepackage{makecell}
\usepackage{subcaption}

\usepackage[backref=page]{hyperref}
\renewcommand*{\backref}[1]{}
\renewcommand*{\backrefalt}[4]{%
    \ifcase #1 (Not cited.)%
    \or        (Cited on page~#2.)%
    \else      (Cited on pages~#2.)%
    \fi}

\newcommand{\R}{\mathbb{R}}
\newcommand{\N}{\mathbb{N}}

\DeclareMathOperator*{\argmin}{arg\,min}
\DeclareMathOperator*{\argmax}{arg\,max}

\newcommand{\cl}{\mathrm{cl}~}
\newcommand{\dom}{\mathrm{dom}~}
\newcommand{\tv}{\mathrm{TV}}
\newcommand{\ri}{\mathrm{ri\ }}
\newcommand{\interior}{\mathrm{int\ }}

\newcommand{\Rn}{\mathbb{R}^n}
\newcommand{\ubar}{\bar{u}}
\newcommand{\pbar}{\bar{p}}
\newcommand{\uvar}{u}
\newcommand{\pvar}{p}
\newcommand{\ind}{I} 
\newcommand{\conv}{\square} 

\newcommand{\imageq}{v}
\newcommand{\imageu}{u}
\newcommand{\icHJS}{J_S}

\newcommand{\domF}{D_F}
\newcommand{\Sstar}{K}

\newcommand{\prior}{J}

\newcommand{\gmRn}{\Gamma_0(\R^n)}

\newcommand{\Vconstraint}{V}

\newcommand{\vz}{v^{(0)}}
\newcommand{\vk}{v^{(k)}}
\newcommand{\xk}{x^{(k)}}
\newcommand{\tk}{t_k}
\newcommand{\dk}{d^{(k)}}

\newcommand{\vkp}{v^{(k+1)}}
\newcommand{\wkp}{w^{(k+1)}}
\newcommand{\ykp}{y^{(k+1)}}
\newcommand{\ukp}{u^{(k+1)}}
\newcommand{\wk}{w^{(k)}}
\newcommand{\yk}{y^{(k)}}

\newcommand{\fBreg}{f}
\newcommand{\gBreg}{g}

\begin{document}

\title{On Hamilton-Jacobi PDEs and image denoising models with certain non-additive noise\thanks{Research supported by NSF DMS 1820821.
Authors' names are given in last/family name alphabetical order.
}
}

\author{J\'er\^ome Darbon         \and
        Tingwei Meng \and
        Elena Resmerita
}

\institute{J\'er\^ome Darbon \at
              Division of Applied Mathematics, Brown University \\
              \email{jerome\_darbon@brown.edu} 
        \and
        Tingwei Meng
        \at
      Division of Applied Mathematics, Brown University \\
      \email{tingwei\_meng@brown.edu}
    \and
        Elena Resmerita \at
        Institute of Mathematics,        Alpen-Adria Universit\"at Klagenfurt,        Universit\"atsstrasse 65--67,     9020 Klagenfurt, Austria\\
      \email{elena.resmerita@aau.at}           
}

\date{Received: date / Accepted: date}

\maketitle

\begin{abstract}
We consider image denoising problems formulated as variational problems. It is known that Hamilton-Jacobi PDEs govern the solution of such optimization problems when the noise model is additive. In this work, we address certain non-additive noise models and show that they are also related to Hamilton-Jacobi PDEs.
These findings allow us to establish new connections between additive and non-additive noise imaging models. 
Specifically, we study how the solutions to these optimization problems depend on the parameters and the observed images. We show that the optimal values are ruled by some Hamilton-Jacobi PDEs, while the optimizers are characterized by the spatial gradient of the solution to the Hamilton-Jacobi PDEs. Moreover, we use these relations to investigate the asymptotic behavior of the variational model as the parameter goes to infinity, that is, when the influence of the noise vanishes.
With these connections, some non-convex models for non-additive noise can be solved by applying convex optimization algorithms to the equivalent convex models for additive noise. Several numerical results are provided for denoising problems with Poisson noise or multiplicative noise.
\end{abstract}

\section{Introduction}

Image denoising  is a fundamental topic in imaging sciences. In this paper, we  consider only digital images which lie in the space of matrices with $n_1$ rows and $n_2$ columns, identified with the Euclidean space $\R^n$ where $n=n_1n_2$ is the number of pixels.
Variational approaches for image denoising have been quite popular and successful \cite{vese2015}.
Given an observed digital image in $\Rn$, a variational approach provides a restored image by solving an optimization problem. For instance, the well-known Rudin–Osher–Fatemi (ROF) model~\cite{rudin1992nonlinear}  approximates
the desired image by the solution $\bar{\imageu}$ of 
\begin{equation}\label{eqt: intro_rof}
    \min_{\imageu\in\R^n}\left\{\tv(\imageu) + \frac{1}{2t}\|x-\imageu\|^2\right\},
\end{equation}
where $t$ is a positive parameter, $x$ is the observed image, $\|\cdot\|$ denotes the Euclidean norm in $\Rn$, and $\tv(\cdot)$ is a discrete total variation seminorm playing the role of a regularization (penalty) term. 
The data fidelity  $\frac{1}{2t}\|x-\imageu\|^2$ in~\eqref{eqt: intro_rof} corresponds to additive Gaussian noise that corrupts the original image. The positive parameter $t$ balances the data fidelity term and the regularization term, in order to ensure accurate and stable reconstructions.
In this paper, the function $\tv(\cdot)$ denotes the anisotropic total variation,
which is defined by
\begin{equation}
    \tv(\imageu) = \sum_{i=1}^{n_1-1}\sum_{j=1}^{n_2} |\imageu_{i+1,j} - \imageu_{i,j}| + \sum_{i=1}^{n_1}\sum_{j=1}^{n_2-1} |\imageu_{i,j+1} - \imageu_{i,j}|,
\end{equation}
where $\imageu$ is a two dimensional digital image with $n_1$ rows and $n_2$ columns, and $\imageu_{i,j}$ denotes its value on the pixel in the $i$-th row and $j$-th column.

Total variation is a widely used regularization term in imaging sciences.
Other than the  total variation, many regularization terms have been proposed in the literature for imaging applications, e.g., non-local total variation \cite{peyre2008} and  $\ell^1$-based  \cite{Chen_Donoho} penalizing functions. In general, a variational model for image denoising problems with additive Gaussian noise reads
\begin{equation}\label{eqt: intro_Gaussian_variation}
    \min_{\imageu\in\R^n}\left\{\prior(\imageu) + \frac{1}{2t}\|x-\imageu\|^2\right\},
\end{equation}
where $\prior$ encodes some knowledge of the image to reconstruct.

The connection between the variational formulation~\eqref{eqt: intro_Gaussian_variation} and certain Hamilton-Jacobi (HJ) PDE has been observed and studied in~\cite{Darbon15}. To be specific, under some assumptions on  $\prior$, the optimal value of this variational problem defined by
\begin{equation*}
    \R^n\times (0,+\infty)\ni (x,t)\mapsto S(x,t) := \min_{\imageu\in\R^n}\left\{\prior(\imageu) + \frac{1}{2t}\|x-\imageu\|^2\right\},
\end{equation*}
solves the following HJ PDE
\begin{equation*}
    \begin{dcases}
        \frac{\partial S}{\partial t}(x,t)+\frac{1}{2}\|\nabla_{x}S(x,t)\|^2 = 0, & \text{if } x\in \R^n, \,\, t>0,\\
        S(x,0)=\prior(x), & \text{if }x\in\mathbb{R}^{n}.
    \end{dcases}
\end{equation*}
The observed image $x$ and the parameter $t$ in the variational model~\eqref{eqt: intro_Gaussian_variation} correspond to the spatial variable and time variable in the HJ PDE, respectively.
Moreover, the minimizer $\bar{\imageu}$ in~\eqref{eqt: intro_Gaussian_variation} is related to the spatial gradient of the solution $S$ by
\begin{equation*}
    \bar{\imageu} = x - t\nabla_x S(x,t).
\end{equation*}

Note that the results in~\cite{Darbon15} also hold for a variational model involving general additive noise, i.e.,
\begin{equation}\label{eqt: intro_variational_additive}
    \min_{\imageu\in\R^n}\{\prior(\imageu) +  g\left(t, x - \imageu\right)\},
\end{equation}
where $J$ is the penalty, 
$g(t,x-u)$ is the data fidelity, and $t$ is a positive parameter.
If  $g(t, x - \imageu)$ can be written as $tH^*\left(\frac{x-\imageu}{t}\right)$ for some convex function $H$ (where $H^*$ denotes the Legendre-Fenchel transform of $H$), then the minimal value in~\eqref{eqt: intro_variational_additive}, denoted by $S(x,t)$, solves the following HJ PDE 
\begin{equation} \label{eqt: intro_HJ_S}
    \begin{dcases}
        \frac{\partial S}{\partial t}(x,t)+H(\nabla_{x}S(x,t)) = 0, & \text{if }x\in\R^n, \,\, t>0,\\
        S(x,0)=J(x), & \text{if }x\in\mathbb{R}^{n},
    \end{dcases}
\end{equation}
under some assumptions, where the initial condition is given by the penalty $J$ in~\eqref{eqt: intro_variational_additive}, and the Hamiltonian $H$ is related to the data fidelity $g$ in~\eqref{eqt: intro_variational_additive}. Moreover, the minimizer $\bar{\imageu}$ in~\eqref{eqt: intro_variational_additive} can be computed using the spatial gradient of the solution $S$ as
\begin{equation*}
    \bar{\imageu} = x - t\nabla H(\nabla_x S(x,t)).
\end{equation*}
Note that some variational models in the literature involving non-Gaussian additive noise can be expressed in the form of~\eqref{eqt: intro_variational_additive} and satisfy the requested assumptions. For instance, the generalized Gaussian noise~\cite{Bouman1993Generalized,Kassam1988Signal} corresponds to the Hamiltonian $H(z)= \frac{1}{\alpha}\|z\|^\alpha$ with $\alpha >1$.
Clearly, relations between PDEs and imaging models provide opportunities to connect the two research areas. Thus, techniques in one area can be applied to solving problems in the other area. For instance, recently, some optimization techniques have been applied to solve high dimensional HJ PDEs \cite{Darbon2016Algorithms}. The reader is referred also to \cite{darbon2019decomposition} for connections between multi-time HJ PDEs and image decomposition problems. 

A natural question that arises here is whether one can  connect variational formulations for   non-additive noise models with certain HJ PDEs. Unfortunately, the results in~\cite{Darbon15} can only handle additive noise or perturbation, and the analysis cannot be straightforwardly applied to the non-additive one. The latter  usually depends on the data, as opposed to the additive case, thus making the investigation more complicated.

In this paper, we propose an answer to this question by establishing links among variational models involving non-additive noise, additive noise, and certain HJ PDEs.

In general, a variational model for denoising problems with non-additive noise reads
\begin{equation}\label{eqt: intro_image_variational}
    \min_{\imageq\in\R^n}\{f(\imageq) + g(t,x,\imageq)\},
\end{equation}
where $f$ is the penalizing term, $g$ is the data fidelity, $t$ is a positive parameter, and $x$ denotes the observed image scaled by $t$.
In this paper, we consider the variational models~\eqref{eqt: intro_image_variational} whose scaled data fidelity part $\frac{1}{t}g(t,x,v)$ can be written as the Bregman distance $D_{H^*}\left(\frac{x}{t}, \imageq\right)$ with respect to a Legendre function $H^*$, and the regularization term $f(v)$ can be written as $J^*(\nabla H^*(\imageq))$ for some convex function $J$, where $H^*$ and $J^*$ are the Legendre-Fenchel transforms of convex functions $H$ and $J$, respectively. 
In other words, we consider the variational formulations which can be expressed as
\begin{equation} \label{eqt: mainresult_var1}
     \min_{\imageq\in\interior\dom H^*} \left\{ J^*(\nabla H^*(\imageq)) + tD_{H^*}\left(\frac{x}{t}, \imageq\right)\right\}.
\end{equation}
This formulation is not as restrictive as it seems to be. 
In general, if a variational model is derived from the Maximum A Posteriori (MAP) estimator, then its data fidelity can be written in the form of $tD_{H^*}\left(\frac{x}{t}, \imageq\right)$ if the likelihood function belongs in a certain class of exponential family. The parameter $t$ is related to the parameter in the exponential family, which has physical meanings for the two models considered in this paper (see Remark~\ref{rem:meaningt_poisson} for the Poisson noise and Remark~\ref{rem:meaningt_multiplicative} for the multiplicative noise). For instance, when $H$ is a quadratic term,~\eqref{eqt: mainresult_var1} corresponds to the variational model for Gaussian noise with regularization term $J^*$. Therefore, the variational problem~\eqref{eqt: mainresult_var1} considered in this paper is a commonly used model in imaging sciences and statistical learning, which includes also the Poisson noise   and the multiplicative noise instances when $D_{H^*}\left(\frac{x}{t}, \imageq\right)$ is the Kullback-Leibler distance and the Itakura-Saito distance, respectively.  In the sequel, we will provide a brief literature review of variational models involving Poisson  noise and multiplicative noise.

The former type of noise is usually considered in the low light environment  such as  in medical imaging, astronomical imaging and microscopy. In this case, the data fidelity in~\eqref{eqt: intro_image_variational} is usually chosen to be
\begin{equation}\label{eqt: intro_fidelity_Poisson}
    g(t,x,\imageq) =t \sum_{i=1}^n \left(\imageq_i - \frac{x_i}{t}\log \imageq_i\right), \quad \forall x=(x_1,\dots, x_n)\in\Rn, \imageq = (\imageq_1,\dots, \imageq_n)\in\Rn_+, t>0.
\end{equation}
The reconstructed image is the minimizer $\bar{\imageq}$ of the variational model~\eqref{eqt: intro_image_variational} with this data-fitting term, where $\frac{x}{t}$ is the observed image and $t$ is the exposure time of the sensor (see Remark~\ref{rem:meaningt_poisson}).
There are plenty of approaches for denoising problems involving Poisson noise. For instance, one can perform some transformations  that reduce Poisson noise to  Gaussian noise -  see, e.g., the Ascombe transform  \cite{anscombe}.  However, such a procedure has its limitations, as it can be efficiently applied only in the case of high signal-to-noise ratio. The reader is referred also to \cite{giryes_elad} which proposes a remedy in the situation of low signal-to-noise ratio.  Iterative methods for denoising and deblurring  images perturbed by Poisson noise can be found in \cite{bertero_2009,poisson_primal_dual}. See also \cite{benning_phd,jung_etal} for  more general settings in infinite dimension.
Since our focus is on variational models, we recall that quite popular  formulations are the ones based on a total variation penalty  (e.g. \cite{le_chartrand_asaki}) and  on an $\ell^1$-norm (see \cite{Figueiredo10}).   The former model seems to be unstable when dealing with very low intensity values or when choosing the regularization parameter via cross-validation. Thus, a series of papers \cite{poisson_logtv1,poisson_mult_logtv} advocate for employing the so-called logarithmic total variation regularizer, that is the total variation applied to the logarithm. This model overcomes the shortcomings mentioned above, but displays instability and slow convergence when addressing high intensity levels, thus prompting research on hybrid penalties, cf. \cite{poisson_logtv2}. Our study will also consider the logarithmic total variation penalty, as it naturally arises in the context.
 
Multiplicative noise models  have been often employed in Synthetic-Aperture Radar (SAR),  ultrasound,  and laser imaging problems.  
As mentioned above, the Itakura-Saito distance
\begin{equation}\label{eqt: intro_fidelity_multiplicative}
    g(t,x,\imageq) = t\sum_{i=1}^n \left(\log \imageq_i +\frac{x_i}{t\imageq_i}\right), \quad \forall x=(x_1,\dots, x_n)\in\Rn, \imageq = (\imageq_1,\dots, \imageq_n)\in\Rn_+, t>0,
\end{equation}
appears as a natural data-fitting term when assuming a Gamma distribution for the noise (see also details in section~\ref{sec: multiplicative}). The reconstructed image is the minimizer $\bar{\imageq}$ of the variational model~\eqref{eqt: intro_image_variational} with this data fidelity, where $\frac{x}{t}$ is the observed image and $t$ is the number of observations (see Remark~\ref{rem:meaningt_multiplicative}). The nonconvex variational approach consisting of the weighted sum between this data fidelity and the total variation seminorm as a penalty has been proposed in the seminal work \cite{Aubert2008Variational}. This motivated some authors to introduce an exponential transform and additional data-fitting terms in order to obtain convex models, cf. \cite{Shi2008Nonlinear,huang_2009,Jin2010Analysis,dong_2013}.
The paper \cite{rudin_lions_osher} proposed removing the multiplicative noise  by  minimizing the total variation seminorm subject to
constraints on the  mean and the variance of the observed data. They employed an alternative data fidelity, namely
\begin{equation}\label{eqt: intro_fidelity_multiplicative1}
    g(t,x,\imageq) = \sum_{i=1}^n \left(\frac{x_i}{t\imageq_i}-1\right)^2, \quad \forall x\in\Rn, \imageq\in\Rn_+, t>0.
\end{equation}
Actually, the following  seems  to be quite a general data fidelity in the  variational denoising problem with multiplicative noise, according to \cite{Shi2008Nonlinear}:
\begin{equation}\label{eqt: intro_fidelity_multiplicative2}
    g(t,x,\imageq) = \sum_{i=1}^n \left(a\frac{x_i}{t\imageq_i}+\frac{b}{2}\left(\frac{x_i}{t\imageq_i}\right)^2 +c\log v_i\right), \quad \forall x\in\Rn, \imageq\in\Rn_+, t>0,
\end{equation}
where $a,b,c$ are real parameters.
Further interesting approaches can be found in \cite{steidl_teuber}  and  \cite{log_additive}. The former  applied a Douglas-Rachford algorithm for minimizing a functional based on the Kullback-Leibler divergence combined with total variation (hence,  appropriate for a Poisson model as well), while the latter   transformed the multiplicative model into an additive one,  by using a logarithm transform (see also \cite{nikolova}).

\textbf{Contributions.} 
In this paper, we provide novel connections between the variational model~\eqref{eqt: mainresult_var1} for denoising problems with non-additive noise and the following variational model for additive noise
\begin{equation}\label{eqt: mainresult_var2}
     \min_{\imageq\in\R^n} \{J(x-t\imageq) + tH^*(\imageq)\},
\end{equation}
an essential role being played by Fenchel duality.
Our contribution is two-fold.
\begin{itemize}
    \item Theoretically, our results link these models for denoising problems to certain HJ PDEs, thus
    opening  the door to new numerical algorithms solving HJ PDEs or denoising problems.
    On the one hand, all algorithms designed for solving the minimization problem~\eqref{eqt: mainresult_var1} can be used to solve the related HJ PDE in very high dimensions. Recall that it is generally hard to solve such problems using a numerical solver in the numerical PDE field since the dimension is high and the Hamiltonian has state and time dependency. On the other hand, if some algorithms are designed for solving the related HJ PDE efficiently, then the algorithm could be applied to solve the variational problem~\eqref{eqt: mainresult_var1}.
    \item Numerically, we provide algorithms for solving the variational model~\eqref{eqt: mainresult_var1} which may be non-convex.
    Note that the variational model~\eqref{eqt: mainresult_var1} includes denoising models for a large class of non-additive noise such as Poisson noise and multiplicative noise. 
    Therefore, our results provide numerical solvers for variational denoising models involving a large class of non-additive noise with certain non-convex regularization terms.
\end{itemize}

\textbf{Organization of the paper.} In section~\ref{sec: math_bkgd}, we introduce the notations which are used in the sequel.
In section~\ref{sec: main_results}, we present the main results in this paper. To be specific, we state the relation among the two minimization problems~\eqref{eqt: mainresult_var1},~\eqref{eqt: mainresult_var2} and two HJ PDEs.
In section~\ref{sec: Moreau}, we establish a generalized Moreau identity, which is then used to show the equivalence of the two variational formulations~\eqref{eqt: mainresult_var1} and~\eqref{eqt: mainresult_var2} under certain assumptions. In section~\ref{sec: HJPDE}, the connection of the variational models and HJ PDEs is proved, and more properties of the HJ PDEs are studied.
Then, we adapt the theoretical framework to two examples involving Poisson noise in section~\ref{sec: Poisson}, and multiplicative noise in section~\ref{sec: multiplicative}. 
An asymptotic result of the variational models is shown in section~\ref{sec:asymptotic}, which proves the convergence of the minimizer as the parameter $t$ goes to infinity under some assumptions.
Robust algorithms for 
the non-convex model~\eqref{eqt: mainresult_var1} in the cases of Poisson and multiplicative noise are proposed and tested in section \ref{sec:numerics} via solving the associated convex problem~\eqref{eqt: mainresult_var2}.

\section{Mathematical background} \label{sec: math_bkgd}
We summarize in Table~\ref{tab:notations} the notations we will use in this paper. For more details on the related basic notions and results, we refer the reader to Appendix~\ref{appendix:bkgd}.

\newcolumntype{s}{>{\hsize=.7\hsize}X}
\begin{table}[!ht]
\centering
 \caption{Notations used in this paper. Here, we use $C$ to denote a set in $\R^n$, $f$ and $g$ to denote functions from $\R^n$ to $\R\cup\{+\infty\}$ and $x,y,p,d,u$ to denote vectors in $\Rn$.}
 \label{tab:notations}
\begin{tabularx}{\textwidth}{ l|s|X }
\hline
\noalign{\smallskip}
 Notation & Meaning & Definition (under certain assumptions)\\
 \noalign{\smallskip}
 \hline
 \noalign{\smallskip}

 $\left\langle \cdot,\cdot\right\rangle $ & 
 Euclidean scalar product in $\mathbb{R}^{n}$ & $\langle x, y\rangle \coloneqq \sum_{i=1}^n x_iy_i$
 \\
 $\left\Vert \cdot\right\Vert$ &
 Euclidean norm in $\mathbb{R}^{n}$ & $\left\Vert x\right\Vert \coloneqq \sqrt{\langle x, x\rangle}$
 \\
 $\ri C$ & Relative interior of $C$ & The interior of $C$ with respect to the minimal hyperplane containing $C$ in $\R^n$
 \\
 $\dom f$ & Domain of $f$ & $\{x\in \R^n\colon f(x) < +\infty\}$
 \\
 $\gmRn$ & A useful and standard class of convex functions & The set containing all proper, convex, lower semicontinuous functions from $\R^n$ to $\R\cup\{+\infty\}$
 \\
 $\partial f(x)$ & Subdifferential of $f$ at $x$ & $\{p\in\Rn:\ f(y)\geq f(x) + \langle p, y-x\rangle \ \forall y\in\Rn\}$
 \\
 $f^*$ & Legendre-Fenchel transform of $f$ & $f^*(p) \coloneqq \sup_{x\in \Rn} \{\langle p, x\rangle - f(x)\}$
 \\
 $f\square g$ & Inf-convolution of $f$ and $g$ &
 $f\square g(x) := \inf_{u\in\R^n} \{f(u) + g(x-u)\}$
 \\
 $f'_\infty$ & The asymptotic function of $f$ &
 $f'_\infty(d) := \sup_{s>0} \left\{\frac{1}{s}(f(x_0 + sd) - f(x_0))\right\}$
 \\
 $d_\fBreg(x,p)$ & primal-dual Bregman distance with respect to $\fBreg$ & $d_\fBreg(x,p) := \fBreg(x) + \fBreg^*(p) - \langle p,x\rangle$
 \\
 $D_\fBreg(x,u)$ & (primal-primal) Bregman distance with respect to $\fBreg$ & $D_\fBreg(x,u):= \fBreg(x)- \fBreg(u) - \langle \nabla \fBreg(u),x-u\rangle$
 \\
 \noalign{\smallskip}
\hline
 \end{tabularx}
\end{table}

\section{Main results}\label{sec: main_results}
In this section, we summarize our theoretical results. We adopt the following assumption
\begin{itemize}
    \item[(A1)] Assume $J$ and $H$ are functions in $\gmRn$. Further assume that $J$ is 1-coercive, and $H$ is a Legendre function.
\end{itemize}

Under this assumption, there are some properties of the functions $J$ and $H$ which will be repeatedly used later. We collect these properties in the following remark.
\begin{remark}\label{rem: properties_HandJ}
Since $J$ is convex, lower semi-continuous and 1-coercive, then by Proposition~\ref{prop: bkgd_finite_coercive}, $J^*$ is convex, lower semi-continuous with $\dom J^* = \R^n$. 
The function $H^*$ is also a Legendre function by Proposition~\ref{prop: bkgd_Legendre}. Then, $H$ and $H^*$ are differentiable in the interior of their domains. Moreover, if $x\in\partial H(p)$ or $p\in\partial H^*(x)$ holds for some $x,p\in\Rn$, then we have $p\in\interior \dom H$, $x\in\interior\dom H^*$, $x=\nabla H(p)$ and $p=\nabla H^*(x)$ by definition of Legendre functions.
\end{remark}

We consider the minimization problem~\eqref{eqt: mainresult_var1}, 
where $D_{H^*}$ denotes the Bregman distance with respect to $H^*$ ($H^*$ stands for the Legendre-Fenchel transform of $H$), and $t$ is a positive parameter.
We show the connection between  problem~\eqref{eqt: mainresult_var1} and the variational model~\eqref{eqt: mainresult_var2} which is commonly used in denoising problems with additive noise.
We also investigate the connections of these variational problems with two HJ PDEs. The illustration of our main results are shown in Figure~\ref{fig: illustration}. 

\begin{figure}[ht]
\includegraphics[width=\textwidth]{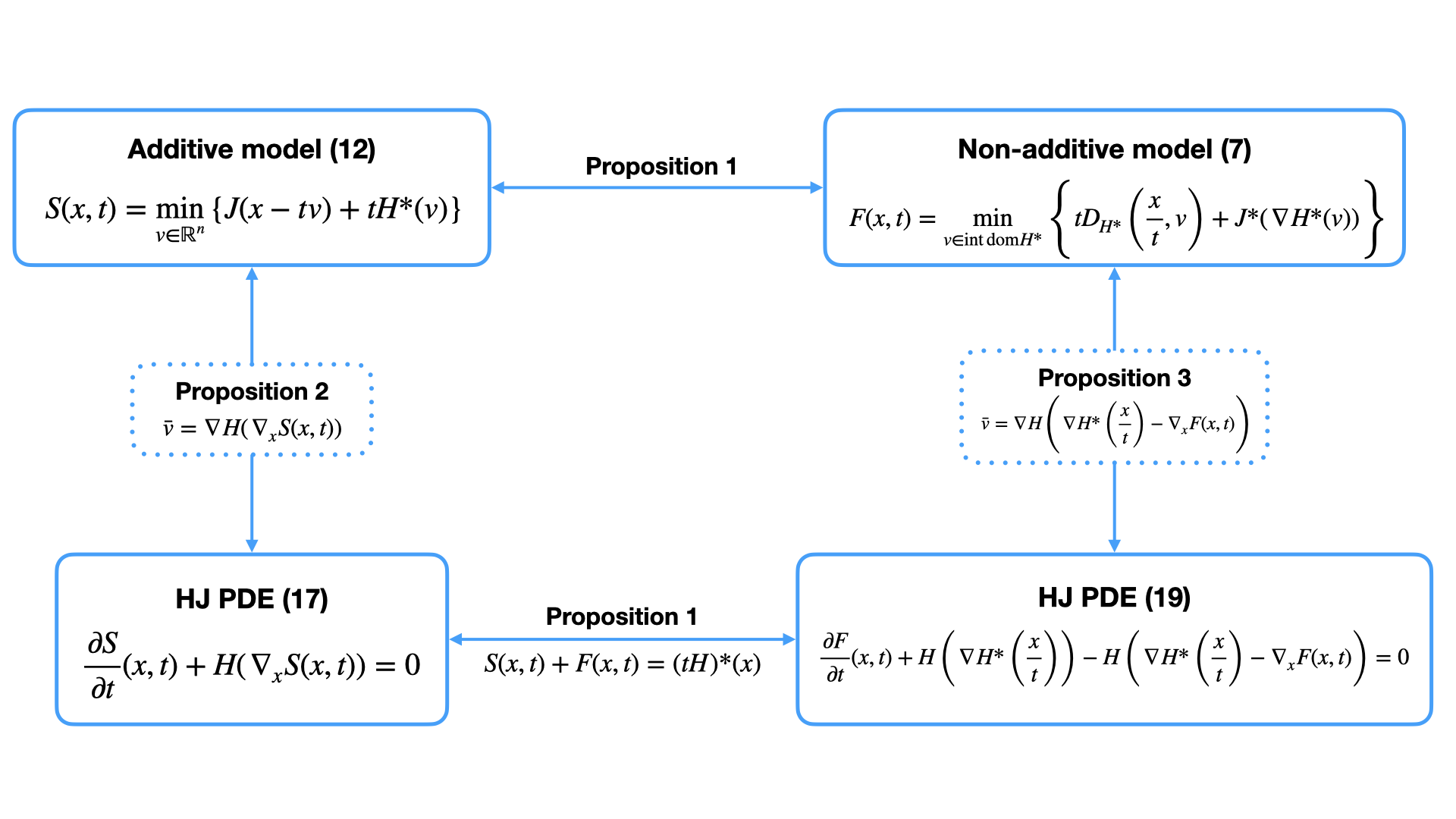}
\caption{An illustration of the relations among the non-additive model~\eqref{eqt: mainresult_var1}, the additive model~\eqref{eqt: mainresult_var2} and the corresponding PDEs~\eqref{eqt: mainresult_HJ_S} and~\eqref{eqt: mainresult_HJ_F}. \label{fig: illustration}}
\end{figure}

The following result shows that under assumption (A1), the minimization problems~\eqref{eqt: mainresult_var1} and~\eqref{eqt: mainresult_var2} are equivalent. The proof and more generalized results are provided in section~\ref{sec: Moreau}.

\begin{proposition} \label{prop: mainresult_Moreau}
Assume that (A1) holds. Let $x\in \R^n$ and $t>0$ satisfy
\begin{equation*}
    x\in \left(\dom J + t\,\interior\dom H^*\right) \cap \left(t\,\dom H^*\right).
\end{equation*}
Then, the minimizer in the minimization problems~\eqref{eqt: mainresult_var1} and~\eqref{eqt: mainresult_var2} exists and is unique. Also, $\bar{\imageq}$ is the minimizer in~\eqref{eqt: mainresult_var1} if and only if it is the minimizer in~\eqref{eqt: mainresult_var2}. In other words, we have
\begin{equation}\label{eqt: prop31_equal_minimizer}
    \argmin_{\imageq\in\,\interior \dom H^*} \left\{tD_{H^*}\left(\frac{x}{t}, \imageq\right) + J^*(\nabla H^*(\imageq))\right\} = \argmin_{\imageq\in\R^n} \{J(x-t\imageq) + tH^*(\imageq)\}.
\end{equation}
Moreover, the minimal values in the two problems~\eqref{eqt: mainresult_var1} and~\eqref{eqt: mainresult_var2} satisfy
    \begin{equation}\label{eqt: prop31_minimal_value}
        \min_{\imageq\in\,\interior \dom H^*} \left\{tD_{H^*}\left(\frac{x}{t}, \imageq\right) + J^*(\nabla H^*(\imageq))\right\}
        + 
        \min_{\imageq\in\R^n} \{J(x-t\imageq) + tH^*(\imageq)\}
        = (tH)^*(x).
    \end{equation}
\end{proposition}

Then, we show the relation of  problems~\eqref{eqt: mainresult_var1} and~\eqref{eqt: mainresult_var2} to two HJ PDEs. 
Generally speaking, the minimal value in~\eqref{eqt: mainresult_var2}, denoted by $S(x,t)$, solves a variant of the HJ PDE~\eqref{eqt: intro_HJ_S}, and the minimal value in~\eqref{eqt: mainresult_var1}, denoted by $F(x,t)$, satisfies another HJ PDE under some assumptions.
Now, we define the functions $S$ and $F$ rigorously. 
The function $S\colon \R^n\times \R\to\R\cup\{+\infty\}$ is defined by
\begin{equation} \label{eqt: mainresult_defS}
    S(x,t) := 
    \begin{dcases}
    \inf_{\imageq\in\R^n} \{J(x-t\imageq) + tH^*(\imageq)\}, &\text{if }x\in\R^n, t>0,\\
    (J^* + \ind_{\dom H})^*(x), &\text{if }x\in\R^n, t=0,\\
    +\infty, &\text{if }x\in\R^n, t<0.
    \end{dcases}
\end{equation}
Moreover, when $0\in \dom J$ holds, we define $F\colon \domF\to \R$ by 
\begin{equation} \label{eqt: mainresult_defF}
    F(x,t) := 
    \begin{dcases}
    \inf_{\imageq\in\,\interior \dom H^*} \left\{tD_{H^*}\left(\frac{x}{t}, \imageq\right) + J^*(\nabla H^*(\imageq))\right\}, &\text{if }t>0,\,\, (x,t)\in \domF,\\
    \ind_{\dom H}^*(x) - (J^* + \ind_{\dom H})^*(x), &\text{if }t = 0,\,\, (x,t)\in\domF,
    \end{dcases}
\end{equation}
where the domain $\domF\subseteq \R^n\times [0,+\infty)$ is defined by
\begin{equation*}
    \domF:= \left\{(x,t)\in\R^n\times (0,+\infty)\colon \frac{x}{t}\in\interior\dom H^*\right\}\bigcup
    (\dom (\ind_{\dom H}^*)\times \{0\}).
\end{equation*}

The following two propositions provide the connections of $S$ and $F$ to two HJ PDEs, respectively. These propositions also provide the representation formula of the minimizer $\bar{\imageq}$ in~\eqref{eqt: mainresult_var1} and~\eqref{eqt: mainresult_var2} using the spatial gradients of the functions $S$ and $F$.
Note that Proposition~\ref{prop: mainresult_HJS} is a generalization of the results in~\cite{Darbon15}, which considers the case when the domain of $H$ is the whole space $\Rn$.
This generalization is needed since the domain of definition of the function $H$ is not $\Rn$ for the Poisson and multiplicative noise 
settings.
The proofs of the following two propositions and more properties of $S$ and $F$ are provided in section~\ref{sec: HJPDE}.

\begin{proposition} \label{prop: mainresult_HJS}
Assume that (A1) holds. Let $S\colon \R^n\times \R\to\R\cup\{+\infty\}$ be the function defined in~\eqref{eqt: mainresult_defS}. Let $\icHJS$ be a function in $\gmRn$ whose Legendre-Fenchel transform $\icHJS^*$ satisfies $\dom \icHJS^* = \cl\dom H$ and $\icHJS^* = J^*$ in the domain of $\icHJS^*$.
Then, $S$ solves the following HJ PDE 
    \begin{equation} \label{eqt: mainresult_HJ_S}
        \begin{dcases}
            \frac{\partial S}{\partial t}(x,t)+H(\nabla_{x}S(x,t)) = 0, & \text{if }(x,t)\in \interior\dom S,\\
            S(x,0)=\icHJS(x), & \text{if }x\in\mathbb{R}^{n}.
        \end{dcases}
    \end{equation}
Moreover, the minimizer $\bar{\imageq}$ in~\eqref{eqt: mainresult_var2} satisfies
\begin{equation}\label{eqt: mainresult_vbar_S}
    \bar{\imageq} = \nabla H(\nabla_x S(x,t)),
\end{equation}
for all $t>0$ and $x\in \dom J + t\,\interior\dom H^*$.
\end{proposition}

\begin{proposition} \label{prop: mainresult_HJF}
Assume that (A1) holds, and also assume $0\in \dom J$. 
Let $F$ be the function defined by~\eqref{eqt: mainresult_defF}.
Then, in the interior of its domain $\domF$, $F$ is continuously differentiable and satisfies the following differential equation
\begin{equation}\label{eqt: mainresult_HJ_F}
    \frac{\partial F}{\partial t}(x,t)
    + H\left(\nabla H^*\left(\frac{x}{t}\right)\right) - H\left(\nabla H^*\left(\frac{x}{t}\right) - \nabla_x F(x,t)\right) = 0, \quad\forall (x,t)\in \interior\dom F.
\end{equation}
Also, $F$ satisfies
\begin{equation}\label{eqt: mainresult_HJ_F_continuity}
    \lim_{t\to 0^+} F(x+td, t)= F(x,0), \quad \forall d\in\interior\dom H^*, \,x\in \dom (\ind_{\dom H}^*).
\end{equation}
Moreover, the minimizer $\bar{v}$ in~\eqref{eqt: mainresult_var1} satisfies
\begin{equation}\label{eqt: mainresult_barv_F}
    \bar{v} = \nabla H\left(\nabla H^*\left(\frac{x}{t}\right) - \nabla_x F(x,t)\right),
\end{equation}
for all $t>0$ and $x\in t\,\interior\dom H^*$.
\end{proposition}

\begin{remark}
The function $S$ is a convex lower semi-continuous function, whose subgradient at $(x,t)\in\Rn\times (0,+\infty)$, whenever exists, satisfies the HJ PDE~\eqref{eqt: mainresult_HJ_S}. 
The uniqueness of such a solution to the HJ PDE~\eqref{eqt: mainresult_HJ_S} is given in~\cite[Corollary~4.1]{darbon2019decomposition}. However, it is unclear whether the other function $F$ is the unique solution to the differential equation~\eqref{eqt: mainresult_HJ_F} with initial condition~\eqref{eqt: mainresult_HJ_F_continuity}. Proposition~\ref{prop: mainresult_HJF} only shows that $F$ is a candidate for solving this PDE. We do not know which functional space should be considered for proving existence and uniqueness of the PDE~\eqref{eqt: mainresult_HJ_F} in an appropriate sense (e.g., in the sense of viscosity solution).

\end{remark}

\section{The generalized Moreau identity} \label{sec: Moreau}

In this section, we provide the generalized Moreau identity and the proof of Proposition~\ref{prop: mainresult_Moreau}. 
Recall that the Moreau identity (decomposition) for a proper lower semi-continuous convex function $\varphi\colon \R^n\to \R\cup\{+\infty\}$ can be formulated as follows~\cite{HiriartUrruty1989Moreau,Moreau1965Proximite}:
$$M_\varphi(x)+M_{\varphi^*}(x)=\frac{\|x\|^2}{2},\quad \forall x\in\R^n,$$
where $M_\varphi$ denotes the Moreau envelope of any function $\varphi\in \gmRn$ which is defined by
\begin{equation*}
M_\varphi(x):=\inf_{y\in\R^n}\left\{\varphi(y)+\frac{1}{2}\|x-y\|^2\right\}=\varphi\square \frac{\|\cdot\|^2}{2}(x), \quad \forall x\in\Rn.
\end{equation*}

Note that the two minimization problems in~\eqref{eqt: prop31_equal_minimizer} are generalizations of the two minimization problems in the definitions of $M_\varphi$ and $M_{\varphi^*}$. To be specific, when we set $\varphi = J$, $H=\frac{1}{2}\|\cdot\|^2$ and $t=1$, the conclusion in Proposition~\ref{prop: mainresult_Moreau} coincides with the Moreau identity. 

In the sequel, we aim at studying a generalization of this decomposition, that will enable us to link the variational models~\eqref{eqt: mainresult_var1} and~\eqref{eqt: mainresult_var2} and prove Proposition~\ref{prop: mainresult_Moreau}.
In the literature, a generalized version of Moreau identity is proposed in~\cite{Combettes2013Moreau}, which works for general Banach spaces. Here, we present a simpler version that works for Euclidean spaces.
The functions $tH$ and $J$ respectively correspond to $f^*$ and $\varphi$ in~\cite[Theorem~1]{Combettes2013Moreau},
but our assumptions are slightly different from the ones in~\cite[Theorem~1]{Combettes2013Moreau} (we do not assume $x\in t\,\interior\dom H^*$ in this paper). For completeness, we include our results as well as the proof.

To study the generalization of Moreau decomposition, we consider the following two optimization problems:
\begin{equation}\label{eqt: lemma_u_opt}
        \ubar = \argmin_{\uvar\in\R^n} \left\{t H^*\left(\frac{x-\uvar}{t} \right) + J(\uvar)\right\},
    \end{equation}
    and
\begin{equation}
\label{eqt: lemma_p_opt}
\pbar = \argmax_{\pvar\in\R^n} \{\langle \pvar,x \rangle - tH(\pvar) - J^*(\pvar)\}.
\end{equation}
We will prove later that they are equivalent to the minimization problems in~\eqref{eqt: prop31_equal_minimizer} after some change of variables. 
Moreover, when we set $\varphi=J$, $H=\frac{1}{2}\|\cdot\|^2$ and $t=1$, problems~\eqref{eqt: lemma_u_opt} and~\eqref{eqt: lemma_p_opt} are also equivalent to the two minimization problems in the Moreau identity.

In the following proposition, the generalized Moreau identity is proved for  problems~\eqref{eqt: lemma_u_opt} and~\eqref{eqt: lemma_p_opt} via Fenchel duality.

\begin{proposition}\label{proposition: existence_uniqueness}
Assume that (A1) holds, and also assume $t>0$ and $x\in \dom J + t\,\interior\dom H^*$. Then, there exist unique vectors $\bar{u}$ and $\bar{p}$ in $\R^n$ solving the optimization problems~\eqref{eqt: lemma_u_opt} and~\eqref{eqt: lemma_p_opt}, respectively.
Moreover, the following two statements are equivalent:
\begin{enumerate}
\item[(a)] The vectors $\ubar$ and $\pbar$ 
solve the optimization problems~\eqref{eqt: lemma_u_opt} and~\eqref{eqt: lemma_p_opt}, respectively.
\item[(b)]  There hold
\begin{equation}\label{eqt: lemma_i} 
    x  = \ubar + t \nabla H (\pbar), \quad\text{ and }\quad 
    J(\ubar) + J^*(\pbar) - \langle \pbar, \ubar\rangle = 0.
\end{equation}
\end{enumerate}
Additionally, the optimal values in~\eqref{eqt: lemma_u_opt} and~\eqref{eqt: lemma_p_opt} are the same, i.e., we have
\begin{equation}\label{eqt: lem_equality_optimal_value}
    \min_{\uvar\in\R^n} \left\{t H^*\left(\frac{x-\uvar}{t} \right) + J(\uvar)\right\}
    = \max_{\pvar\in\R^n} \{\langle \pvar,x \rangle - tH(\pvar) - J^*(\pvar)\}.
\end{equation}
\end{proposition}
\begin{proof}
We apply~\cite[Remark~III.4.2]{ekeland1976}, where we set $\Lambda$ to be the identity map, $G\colon u\mapsto t H^*\left(\frac{x-\uvar}{t} \right)$ and $F=J$ to hold. 
Note that in this proof, the notation $F$ is used as  in~\cite{ekeland1976}, whose meaning is different from what we defined in~\eqref{eqt: mainresult_defF}.
By~\cite[Proposition~E.1.3.1]{hiriart2004}, we have
\begin{equation*}
    G^*(p) = \langle x,p\rangle + tH(-p), \quad \forall p\in\Rn.
\end{equation*}
Under this setup, the primal problem $\inf\mathcal{P}$ can be written as the problem in~\eqref{eqt: lemma_u_opt}, and the dual problem $\sup \mathcal{P}^*$ (see (4.18) in Chap.III in~\cite{ekeland1976}) can be written as
\begin{equation*}
    \sup_{p\in\Rn}\left\{-F^*(p) - G^*(-p)\right\} 
    = \sup_{p\in\Rn}\left\{-J^*(p) + \langle x,p\rangle - tH(p)\right\}, 
\end{equation*}
which is the problem in~\eqref{eqt: lemma_p_opt}. Then, by Theorem~4.2 and Remark~4.2 in Chapter III in~\cite{ekeland1976}, to prove the existence of $\ubar$ and $\pbar$, it suffices to prove the following two statements
\begin{equation}\label{eqt: prop43_pf_cond1}
    \lim_{\|u\|\to+\infty} \left\{t H^*\left(\frac{x-\uvar}{t} \right) + J(\uvar)\right\} = +\infty,
\end{equation}
and 
\begin{equation}\label{eqt: prop43_pf_cond2}
    \exists u_0\in\Rn \text{ s.t. }J(u_0)<+\infty,\, 
    G(u_0)<+\infty,
    \text{  $G$ being continuous at $u_0$.}
\end{equation}
Note that~\eqref{eqt: prop43_pf_cond1} holds since $J$ is assumed to be 1-coercive. It remains to check~\eqref{eqt: prop43_pf_cond2}. By assumption, $x\in \dom J + t\,\interior\dom H^*$ holds, and hence there exists $u_0\in\dom J$ such that $\frac{x-u_0}{t}$ is in the interior of $\dom H^*$, which implies the continuity of $H^*$ at $\frac{x-u_0}{t}$. As a result,~\eqref{eqt: prop43_pf_cond2} holds for this $u_0$. The existence of the optimizers $\ubar$ and $\pbar$ follows.
Moreover, Theorem~4.2 and Remark~4.2 in Chapter III in~\cite{ekeland1976} imply~\eqref{eqt: lem_equality_optimal_value}.

Now, we prove the equivalence of the two statements (a) and (b).
By~\eqref{eqt: lem_equality_optimal_value} and~\cite[Proposition~III.4.1 and Remark~III.4.2]{ekeland1976}, $\ubar$ and $\pbar$ satisfy~\eqref{eqt: lemma_u_opt} and~\eqref{eqt: lemma_p_opt} if and only if there hold
\begin{equation}\label{eqt: prop43pf_optimality_FG}
-\pbar\in \partial G(\ubar) \text{ and }
\pbar\in \partial F(\ubar).
\end{equation}
As it is proved above, there exists $u_0\in\Rn$ such that $\frac{x-u_0}{t}$ is in the interior of $\dom H^*$. As a result,~\cite[Theorem~XI.3.2.1 and (3.2.2)]{hiriart2004} hold, by which we have
\begin{equation*}
    \partial G(u) = 
    \begin{dcases}
    -\partial H^*\left(\frac{x-u}{t}\right), & \text{if }u\in x-t\,\interior\dom H^*,\\
    \emptyset, & \text{otherwise}.
    \end{dcases}
\end{equation*}
Therefore,~\eqref{eqt: prop43pf_optimality_FG} is equivalent to
\begin{equation}\label{eqt: prop43pf_optimality}
    \pbar\in \partial H^*\left(\frac{x-\ubar}{t}\right),
    \text{ and } \pbar\in\partial J(\ubar).
\end{equation}
Since $H$ is a Legendre function, the first inclusion in~\eqref{eqt: prop43pf_optimality} holds
if and only if $\frac{x-\ubar}{t}=\nabla H(\pbar)$ is verified,
which is equivalent to the first equality in~\eqref{eqt: lemma_i}. Also, by Proposition~\ref{prop: bkgd_conjugate_equal_inverse}, the second inclusion in~\eqref{eqt: prop43pf_optimality} is equivalent to the second equality in~\eqref{eqt: lemma_i}. Therefore, the equivalence of statements (a) and (b) follows. 

It remains to prove the uniqueness of the two optimizers $\ubar$ and $\pbar$. Let $\ubar$ be any minimizer in~\eqref{eqt: lemma_u_opt}, and $\pbar$ be any maximizer in~\eqref{eqt: lemma_p_opt}. From the argument above, $\ubar$ and $\pbar$ satisfy~\eqref{eqt: prop43pf_optimality}, which implies that $\pbar\in\interior\dom H$ and $\frac{x-\ubar}{t}\in\interior \dom H^*$ hold since $H$ is Legendre (see Remark~\ref{rem: properties_HandJ}). Then, the uniqueness of $\ubar$ and $\pbar$ follows since $H$ and $H^*$ are strictly convex in the interior of their domains. 
\qed
\end{proof}

\begin{remark}
    Note that the 1-coercivity (supercoercivity) of $J$ is the main tool in showing existence of the above solutions. However, one can ensure the existence of the optimizers in~\eqref{eqt: lemma_u_opt} and~\eqref{eqt: lemma_p_opt} under less restrictive assumptions, namely conditions that yield~\eqref{eqt: prop43_pf_cond1}.
\end{remark}

\begin{remark} \label{rem: domS_interior}
Let $t>0$.
Since $H^*$ is Legendre (see Remark~\ref{rem: properties_HandJ}), the interior of $\dom H^*$ is non-empty, and hence the set $\dom J + t\,\dom H^*$ also has non-empty interior. Moreover, by~\cite[Proposition~A.2.1.11,~A.2.1.12]{hiriart2004} and the discussion after these propositions,
we have
\begin{equation}\label{eqt: rem_int_domF}
\begin{split}
\interior\left(\dom J + t\,\dom H^*\right) &= \ri\left(\dom J + t\,\dom H^*\right) = \ri\dom J + t\,\ri \dom H^*\\
&= \ri\dom J + t\,\interior \dom H^*  \subseteq \dom J + t\,\interior \dom H^*.
\end{split}
\end{equation}
Also, we have $\dom J + t\,\interior \dom H^*\subseteq \dom J + t\,\dom H^*$.
Since the set $\dom J + t\,\interior \dom H^*$ is the sum of two sets in $\R^n$ where one of them is open, then $\dom J + t\,\interior \dom H^*$ is also an open set, which implies 
\begin{equation}\label{eqt:rem43inclusion2}
    \dom J + t\,\interior \dom H^*  \subseteq \interior\left(\dom J + t\,\dom H^*\right).
\end{equation}
Combining~\eqref{eqt: rem_int_domF} and~\eqref{eqt:rem43inclusion2}, we get
\begin{equation}\label{eqt: rem47_interior}
    \interior\left(\dom J + t\,\dom H^*\right) = \dom J + t\,\interior \dom H^*.
\end{equation}
Therefore, the assumption $x\in \dom J + t\,\interior\dom H^*$ in Proposition~\ref{proposition: existence_uniqueness} is equivalent to $x\in \interior\left(\dom J + t\,\dom H^*\right)$.
\end{remark}

\begin{remark} \label{rem: weaker_assumption_x}
Let $t>0$. 
According to Proposition~\ref{prop: bkgd_sum_Legendre_infconv}, 
the optimal value in~\eqref{eqt: lemma_u_opt} is finite if and only if $x\in\dom J + t\,\dom H^*$ holds. Moreover, whenever $x\in\dom J + t\,\dom H^*$ holds, the minimizer $\ubar$ exists, and we have
\begin{equation*}
    \min_{\uvar\in\R^n} \left\{t H^*\left(\frac{x-\uvar}{t} \right) + J(\uvar)\right\}
    = \sup_{\pvar\in\R^n} \{\langle \pvar,x \rangle - tH(\pvar) - J^*(\pvar)\}.
\end{equation*}
Note that this feasible set for $x$ is slightly larger than the set $\dom J + t\,\interior\dom H^*$ in the assumption of Proposition~\ref{proposition: existence_uniqueness}, which is actually the interior of the larger set $\dom J + t\,\dom H^*$ by Remark~\ref{rem: domS_interior}. 
We employ the stronger assumption on $x$ in Proposition~\ref{proposition: existence_uniqueness}, in order to have the existence of the maximizer $\pbar$ in~\eqref{eqt: lemma_p_opt}. A counterexample where the maximizer $\pbar$ does not exist under the weaker assumption $x\in\dom J + t\,\dom H^*$ is when $x=0$, $J=\ind_{\{0\}}$ and $H(p)=e^p$ for all $p\in\R$.
\end{remark}

\bigbreak
Now, we apply Proposition~\ref{proposition: existence_uniqueness} to prove our main result in Proposition~\ref{prop: mainresult_Moreau}.

\textbf{Proof of Proposition~\ref{prop: mainresult_Moreau}.}
Let $\imageq\in \interior\dom H^*$. Since $H^*$ is Legendre (see Remark~\ref{rem: properties_HandJ}), $H^*$ is differentiable at $\imageq$. Let $p = \nabla H^*(\imageq)$, which implies $\imageq = \nabla H(p)$ as also $H$ is Legendre. By straightforward computation, we obtain
\begin{equation}\label{eqt: prop32pf_Bregman1}
\begin{split}
    tD_{H^*}\left(\frac{x}{t}, \imageq\right) + J^*(\nabla H^*(\imageq))
    &= tD_{H^*}\left(\frac{x}{t}, \nabla H(p)\right) + J^*(p)
    = D_{(tH)^*}\left(x, t\nabla H(p)\right) + J^*(p),
\end{split}
\end{equation}
where the second equality holds by Proposition~\ref{prop: Breg_scaling} whose assumption is satisfied since $\frac{x}{t}$ and $\nabla H(p)$ are both in $\dom H^*$.
Since $tH$ is also Legendre, the function $(tH)^*$ is differentiable at $t\nabla H(p) = \nabla (tH)(p)$, and the gradient equals $\nabla (tH)^*\left(\nabla (tH)(p)\right) = p$.
As a result, the assumption of Proposition~\ref{prop: Breg_equal_twodef} is satisfied, which implies 
\begin{equation}\label{eqt: prop32pf_Bregman2}
\begin{split}
    D_{(tH)^*}\left(x, t\nabla H(p)\right) + J^*(p)
    &= d_{(tH)^*}\left(x, p\right) + J^*(p)= d_{tH}\left(p, x\right) + J^*(p) \\
    &= tH(p) - \langle p,x\rangle + J^*(p) + tH^*\left(\frac{x}{t}\right),
\end{split}
\end{equation}
where the second equality holds by~\eqref{eqt: equal_Bregd}.
By~\eqref{eqt: prop32pf_Bregman1} and~\eqref{eqt: prop32pf_Bregman2}, we get
\begin{equation} \label{eqt: prop32pf_changeofvar_equality}
    tD_{H^*}\left(\frac{x}{t}, \imageq\right) + J^*(\nabla H^*(\imageq))
    =tH(p) - \langle p,x\rangle + J^*(p) + tH^*\left(\frac{x}{t}\right),
\end{equation}
and hence there holds
\begin{equation}\label{eqt: prop31pf_changeofvar}
\begin{split}
    \inf_{\imageq\in\,\interior \dom H^*} \left\{tD_{H^*}\left(\frac{x}{t}, \imageq\right) + J^*(\nabla H^*(\imageq))\right\}
    &\geq \inf_{p\in\Rn}\left\{tH(p) - \langle p,x\rangle + J^*(p) + tH^*\left(\frac{x}{t}\right)\right\} \\
    &= \inf_{p\in\Rn}\{tH(p) - \langle p,x\rangle + J^*(p)\} + tH^*\left(\frac{x}{t}\right),
\end{split}
\end{equation}
where the last equality holds since $tH^*\left(\frac{x}{t}\right)$ is a real number.
Note that the minimization problem on the right hand side of~\eqref{eqt: prop31pf_changeofvar} is equivalent to the maximization problem in~\eqref{eqt: lemma_p_opt} by taking the negative sign.
Then, by Proposition~\ref{proposition: existence_uniqueness}, the minimizer on the right hand side of~\eqref{eqt: prop31pf_changeofvar} exists and is unique. We denote this unique minimizer by $\pbar$. Moreover, Proposition~\ref{proposition: existence_uniqueness} implies that $\pbar$ is in the interior of $\dom H$, and $H$ is differentiable at $\pbar$.
Let $\bar v=\nabla H(\pbar)$, which implies
$\pbar =\nabla H^*(\bar v)$ since $H$ is Legendre.
From the calculation above, we deduce that~\eqref{eqt: prop32pf_changeofvar_equality} holds for $\bar v$ and $\pbar$. Hence, by~\eqref{eqt: prop31pf_changeofvar} and the fact that $\bar v$ is in $\interior \dom H^*$ (according to Remark~\ref{rem: properties_HandJ}), we obtain
\begin{equation}\label{eqt:prop1_pf_eqt37}
\begin{split}
    \inf_{p\in\Rn}\left\{tH(p) - \langle p,x\rangle + J^*(p) + tH^*\left(\frac{x}{t}\right)\right\}
    &= tH(\pbar) - \langle \pbar,x\rangle + J^*(\pbar) + tH^*\left(\frac{x}{t}\right)\\
    &= tD_{H^*}\left(\frac{x}{t}, \bar\imageq\right) + J^*(\nabla H^*(\bar\imageq))\\
    &\geq \inf_{\imageq\in\,\interior \dom H^*} \left\{tD_{H^*}\left(\frac{x}{t}, \imageq\right) + J^*(\nabla H^*(\imageq))\right\}.
\end{split}
\end{equation}
Comparing~\eqref{eqt:prop1_pf_eqt37} with~\eqref{eqt: prop31pf_changeofvar}, we conclude that both inequalities in~\eqref{eqt:prop1_pf_eqt37} and~\eqref{eqt: prop31pf_changeofvar} become equalities,
$\bar v$ is a minimizer of~\eqref{eqt: mainresult_var1}, and there holds
\begin{equation}\label{eqt: prop32pf_equiv_var1}
    \inf_{\imageq\in\,\interior \dom H^*} \left\{tD_{H^*}\left(\frac{x}{t}, \imageq\right) + J^*(\nabla H^*(\imageq))\right\}
    = \inf_{p\in\Rn}\{tH(p) - \langle p,x\rangle + J^*(p)\} + tH^*\left(\frac{x}{t}\right).
\end{equation}
Moreover, if there are two distinct minimizers of~\eqref{eqt: mainresult_var1}, denoted by $\bar v_1$ and $\bar{v}_2$, then with the argument above, we conclude that $\nabla H^*(\bar v_1)$ and $\nabla H^*(\bar v_2)$ are both minimizers of the minimization problem on the right hand side of~\eqref{eqt: prop31pf_changeofvar}. In other words, we have $\nabla H^*(\bar v_1) =\nabla H^*(\bar v_2)=\pbar$, which implies $\bar v_1 = \bar v_2$, since $H$ is Legendre.
Therefore, the minimizer of~\eqref{eqt: mainresult_var1} exists and is unique, which equals $\nabla H(\pbar)$.

On the other hand, by the change of variable $v=\frac{x-u}{t}$, we deduce that 
\begin{equation}\label{eqt: prop32pf_equiv_var2}
    \min_{\uvar\in\R^n} \left\{t H^*\left(\frac{x-\uvar}{t} \right) + J(\uvar)\right\}
    = \min_{\imageq\in\R^n} \{tH^*(\imageq) + J(x-t\imageq) \}.
\end{equation}
Let $\bar{u}$ and $\bar{v}$ be the minimizers of the two optimization problems in~\eqref{eqt: prop32pf_equiv_var2}, respectively. According to the change of variable, we have $\bar{v} = \frac{x-\bar{u}}{t}$, and hence there holds
\begin{equation}\label{eqt: prop32pf_equiv_var2_argmin}
    \frac{1}{t}\left(x - \argmin_{\uvar\in\R^n} \left\{t H^*\left(\frac{x-\uvar}{t} \right) + J(\uvar)\right\}\right)
    = \argmin_{\imageq\in\R^n} \{tH^*(\imageq) + J(x-t\imageq) \}.
\end{equation}
In other words,  problem~\eqref{eqt: mainresult_var2} is related to~\eqref{eqt: lemma_u_opt} via the change of variable.
By Proposition~\ref{proposition: existence_uniqueness}, the minimizer in~\eqref{eqt: lemma_u_opt} exists and is unique, which implies the existence and uniqueness of the minimizer in  problem~\eqref{eqt: mainresult_var2}.
Recall that $\ubar$ is a minimizer in~\eqref{eqt: lemma_u_opt}.
By~\eqref{eqt: prop32pf_equiv_var2_argmin}, the unique minimizer in  problem~\eqref{eqt: mainresult_var2} equals $\frac{x-\ubar}{t}$,
which equals $\nabla H(\pbar)$ by Proposition~\ref{proposition: existence_uniqueness}. Therefore,~\eqref{eqt: prop31_equal_minimizer} is proved.
Moreover,~\eqref{eqt: prop31_minimal_value} follows from~\eqref{eqt: prop32pf_equiv_var1},~\eqref{eqt: prop32pf_equiv_var2} and Proposition~\ref{proposition: existence_uniqueness}.
\qed

\section{Variational models and HJ PDEs} \label{sec: HJPDE}

In this section, we investigate the functions $S$ and $F$ defined in~\eqref{eqt: mainresult_defS} and~\eqref{eqt: mainresult_defF}, respectively. We provide the proofs for Propositions~\ref{prop: mainresult_HJS} and~\ref{prop: mainresult_HJF}, and some other properties. 
By change of variable $u=x-tv$ in the optimization problem in the first line~\eqref{eqt: mainresult_defS}, 
the function $S$ satisfies
\begin{equation} \label{eqt: S_equiv_formula}
    S(x,t) = 
    \begin{dcases}
    \inf_{u\in\R^n} \left\{t H^*\left(\frac{x-u}{t} \right) + J(u)\right\}, &\text{if }x\in\R^n, t>0,\\
    (J^* + \ind_{\dom H})^*(x), &\text{if }x\in\R^n, t=0,\\
    +\infty, &\text{if }x\in\R^n, t<0.
    \end{dcases}
\end{equation}
We show some properties of the function $S$ and its connection with the HJ PDE~\eqref{eqt: mainresult_HJ_S} in the following proposition.

\begin{proposition}  \label{prop: HJPDE}
Assume that (A1) holds. Let $S\colon \R^n\times \R\to\R\cup\{+\infty\}$ be the function defined in~\eqref{eqt: mainresult_defS}. Then, the following statements hold.
\begin{itemize}
    \item[(a)] $S$ is a convex and lower semi-continuous function with respect to the joint variable. The Legendre-Fenchel transform of $S$ is given by
    \begin{equation} \label{eqt: prop_Legendre_S}
        S^*(p, E^-) = J^*(p) + \ind_{\{(p,E^-)\in(\dom H)\times\R\colon E^- + H(p)\leq 0\}}(p,E^-),
    \end{equation}
    for all $p\in\R^n$ and $E^-\in\R$.
    
    \item[(b)] The domain of $S$ equals
    \begin{equation} \label{eqt: prop_pdeb_domS}
    \begin{split}
        \dom S = \left(\bigcup_{t>0} (\dom J + t\,\dom H^*)\times\{t\}\right) \bigcup\left((\dom J + \dom (\ind_{\dom H}^*))\times \{0\}\right).
    \end{split}
    \end{equation}
    Moreover, the interior of $\dom S$ equals
    \begin{equation}\label{eqt: prop51_b_interiordomS}
        \interior\dom S = \bigcup_{t>0} (\dom J + t\,\interior\dom H^*)\times\{t\}.
    \end{equation}
    \item[(c)] Let $(x,t)$ be an arbitrary point in $\interior\dom S$. Then, $S$ is differentiable at $(x,t)$, and its gradient equals
    \begin{equation} \label{eqt: gradS}
        \nabla S(x,t) = (\pbar, -H(\pbar)),
    \end{equation}
    where $\pbar$ is the unique maximizer in~\eqref{eqt: lemma_p_opt} at $(x,t)$.
    
    \item[(d)] $S$ solves the HJ PDE~\eqref{eqt: mainresult_HJ_S},
    where $\icHJS\colon \R^n\to\R\cup\{+\infty\}$ is a function in $\gmRn$ whose Legendre-Fenchel transform $\icHJS^*$ satisfies $\dom \icHJS^* = \cl \dom H$ and $\icHJS^* = J^*$ in the domain of $\icHJS^*$.
    
    \item[(e)] For all $x\in\R^n$, we have $S(x,0)\leq J(x)$. If $x\in\dom J$ satisfies $\partial J(x)\cap \cl \dom H\neq \emptyset$, then there hold $S(x,0)= J(x)$ and $\partial_x S(x,0) = \partial J(x) \cap \cl \dom H$, where $\partial_x S(x,0)$ denotes the subdifferential of the function $y\mapsto S(y,0)$. 
\end{itemize}
\end{proposition}
\begin{proof}
(a) Denote by $\Vconstraint$ the set 
\begin{equation} \label{eqt: Vconstraint}
    \Vconstraint := {\{(p,E^-)\in(\dom H)\times\R\colon E^- + H(p)\leq 0\}}.
\end{equation}
Define a function $\Sstar\colon \R^n\times\R\to \R\cup\{+\infty\}$ by 
\begin{equation} \label{eqt: defSstar}
    \Sstar(p,E^-) = J^*(p) + \ind_\Vconstraint(p,E^-),
\end{equation}
for all $p\in\R^n$ and $E^-\in\R$.
Since $\dom J^*=\R^n$ holds (see Remark~\ref{rem: properties_HandJ}), then $\Sstar$ is well-defined and proper.
Note that $\Sstar$ is convex and lower semi-continuous, i.e., $\Sstar\in \Gamma_0(\R^n\times \R)$, since $J^*$ and $H$ are both convex and lower semi-continuous.
Now, we prove that $\Sstar^* = S$. 

First, consider the case when $x\in\R^n$ and $t>0$. By definition, we have
\begin{equation}\label{eqt: pf_pde_Sstar_tpos}
    \Sstar^*(x,t) = \sup_{p\in\dom H}\left\{\sup_{E^-\in\R} \{\langle p,x\rangle + tE^- - J^*(p) - \ind_\Vconstraint(p,E^-)\}\right\}.
\end{equation}
We calculate the inner supremum as follows,
\begin{equation} \label{eqt: pf_pde_Sstar_Em}
\begin{split}
    \sup_{E^-\in\R} \{\langle p,x\rangle + tE^- - J^*(p) - \ind_\Vconstraint(p,E^-)\}
    &= \sup_{E^-\leq -H(p)} \{\langle p,x\rangle + tE^- - J^*(p) \}\\
    &= \max_{E^-\leq -H(p)} \{\langle p,x\rangle + tE^- - J^*(p) \}\\
    &= \langle p,x\rangle - tH(p)- J^*(p),
\end{split}
\end{equation}
where the first equality is implied by the definition of $V$ in~\eqref{eqt: Vconstraint}, and the second and third equalities hold because the maximizer is given by $E^- = -H(p)$.
Then, combining~\eqref{eqt: pf_pde_Sstar_tpos} and~\eqref{eqt: pf_pde_Sstar_Em}, we obtain
\begin{equation*}
\begin{split}
    \Sstar^*(x,t)
    &= \sup_{p\in\dom H}\{ \langle p,x\rangle - J^*(p) - tH(p)\} 
    = (J^* + tH)^*(x).
\end{split}
\end{equation*}
Recall that the domain of $J^*$ is the whole space $\R^n$. Then, by Proposition~\ref{prop: bkgd_sum_Legendre_infconv}, we have 
\begin{equation*}
    (J^*+tH)^*(x) = (J\conv (tH)^*)(x) = \inf_{u\in\R^n} \left\{t H^*\left(\frac{x-u}{t} \right) + J(u)\right\} = S(x,t),
\end{equation*}
where the second equality holds since $(tH)^* = tH^*(\frac{\cdot}{t})$ (see~\cite[Proposition E.1.3.1]{hiriart2004}).
Therefore, we have proved that $\Sstar^*(x,t) = S(x,t)$ for all $x\in\R^n$ and $t>0$.

Then, we consider the case when $t=0$. For all $x\in\R^n$, we have
\begin{equation*}
\begin{split}
    \Sstar^*(x,0) &= \sup_{p\in\dom H}\left\{\sup_{E^-\in\R} \{\langle p,x\rangle  - J^*(p) - \ind_\Vconstraint(p,E^-)\}\right\}\\
    &= \sup_{p\in\dom H}\{ \langle p,x\rangle - J^*(p)\} \\
    &= (J^* + \ind_{\dom H})^*(x) = S(x,0).
\end{split}
\end{equation*}

It remains to consider the case when $t<0$. In this case, for all $p\in\dom H$, it holds
\begin{equation*}
\sup_{E^-\in\R} \{tE^- - \ind_\Vconstraint(p,E^-)\} = \sup_{E^-\leq -H(p)} tE^-  = +\infty.
\end{equation*}
Hence, for all $x\in\R^n$ and $t<0$, it follows that
\begin{equation*}
\begin{split}
    \Sstar^*(x,t) &= \sup_{p\in\dom H}\left\{\sup_{E^-\in\R} \{\langle p,x\rangle + tE^- - J^*(p) - \ind_\Vconstraint(p,E^-)\}\right\} = +\infty = S(x,t).
\end{split}
\end{equation*}
Therefore, we have proved that $\Sstar^* = S$, which implies  convexity and lower semi-continuity of $S$. Since $\Sstar$ is also convex and lower semi-continuous, we have $S^* = (\Sstar^*)^* = \Sstar$ and thus,~\eqref{eqt: prop_Legendre_S}.

(b) First, for all $t>0$, by definition of $S$ in~\eqref{eqt: mainresult_defS} and definition of inf-convolution, we deduce that
\begin{equation}\label{eqt: prop51pf_b_eq1}
    \dom (x\mapsto S(x,t))
    = \dom J + t\,\dom H^*.
\end{equation}
Now, consider the case when $t=0$. By~\cite[Proposition~B.1.2.5]{hiriart2004}, the lower semi-continuous hull of $J^* + \ind_{\dom H}$ is $J^* + \ind_{\cl\dom H}$, and hence we have
\begin{equation} \label{eqt: prop51pf_S0star}
    (x\mapsto S(x,0))^* = ((J^* + \ind_{\dom H})^*)^* = J^* + \ind_{\cl\dom H}.
\end{equation}
By definition, the function $x\mapsto S(x,0)$ is in $\gmRn$, and hence there holds
\begin{equation*} 
    S(x,0) = ((x\mapsto S(x,0))^*)^* = (J^* + \ind_{\cl\dom H})^*.
\end{equation*}
Since $\dom J^*=\R^n$, 
applying Proposition~\ref{prop: bkgd_sum_Legendre_infconv} to $J^*$ and $\ind_{\cl\dom H}$, we have 
\begin{equation}\label{eqt: pf_pdeb_S0}
    S(x,0) = (J^* + \ind_{\cl\dom H})^*(x) = J\conv \ind_{\cl\dom H}^*(x) = J\conv \ind_{\dom H}^*(x).
\end{equation}
As a result, we obtain
\begin{equation}\label{eqt: prop51pf_b_eq2}
    \dom (x\mapsto S(x,0)) = \dom J + \dom (\ind_{\dom H}^*).
\end{equation}
Then,~\eqref{eqt: prop_pdeb_domS} follows from~\eqref{eqt: prop51pf_b_eq1} and~\eqref{eqt: prop51pf_b_eq2}. 

Now, we calculate the interior of $\dom S$.
Since $\dom S$ is a subset of $\R^n\times [0,+\infty)$, there holds
\begin{equation*}
    \interior\dom S\subseteq \interior \left(\R^n\times [0,+\infty)\right) = \R^n\times (0,+\infty).
\end{equation*}
Also, for all $t>0$, the intersection $(\interior\dom S) \cap (\R^n\times \{t\})$ is open in $\R^n\times\{t\}$ with respect to the subspace topology. As a result, we get
\begin{equation*}
\begin{split}
   (\interior\dom S) \cap (\R^n\times \{t\})\subseteq \ri \left(\dom S\cap (\R^n\times \{t\})\right) &= \left(\interior (\dom J + t\,\dom H^*)\right)\times \{t\}\\
   &= \left(\dom J + t\,\interior \dom H^*\right)\times \{t\},
\end{split}
\end{equation*}
where the first equality holds by~\eqref{eqt: prop51pf_b_eq1} and the definition of relative interior, and the last equality holds by~\eqref{eqt: rem47_interior}.
Therefore, we obtain
\begin{equation} \label{eqt: prop51pf_b_inclusion}
    \interior\dom S = \bigcup_{t>0}\left((\interior\dom S) \cap (\R^n\times \{t\})\right)\subseteq \bigcup_{t>0} (\dom J + t\,\interior\dom H^*)\times\{t\}\subseteq \dom S.
\end{equation}
Then, it remains to prove that the set 
\begin{equation}\label{eqt: prop51pf_b_interiorset}
    \bigcup_{t>0} (\dom J + t\,\interior\dom H^*)\times\{t\}
\end{equation}
is open. Let $(x,t)$ be a point in this set. We have $t>0$, and there exists $v\in \interior\dom H^*$ satisfying $u:=x-tv\in\dom J$.
Then, there exists $\delta>0$ such that the $\delta-$neighborhood of $v$, denoted by $\mathcal{N}_\delta(v)$, is included in $\interior\dom H^*$. Let $\epsilon\in (0,t)$ satisfy $\frac{\epsilon(1+\|v\|)}{t-\epsilon} < \delta$. For all $(y,s)\in\R^n\times\R$ satisfying $\|y-x\|<\epsilon$ and $|t-s|<\epsilon$, we have 
\begin{equation*}
    \left\|\frac{y- u}{s} - v\right\| = \left\|\frac{y- (x-tv)}{s} - v\right\|= \left\|\frac{y-x}{s} + \frac{t-s}{s}v\right\|\leq \frac{\epsilon}{s} + \frac{\epsilon \|v\|}{s}\leq \frac{\epsilon(1+\|v\|)}{t-\epsilon}<\delta.
\end{equation*}
As a result, we have $\frac{y- u}{s}\in\mathcal{N}_\delta(v)$, and hence $\frac{y- u}{s}$ is in $\interior\dom H^*$. Recall that $u$ is in $\dom J$. Then, $(y,s)$ is in the set in~\eqref{eqt: prop51pf_b_interiorset}. Therefore, the set in~\eqref{eqt: prop51pf_b_interiorset} is open, and then~\eqref{eqt: prop51_b_interiordomS} follows from~\eqref{eqt: prop51pf_b_inclusion}.

(c)  
We still use the notations $\Vconstraint$ and $\Sstar$ as in the proof of (a), whose definitions are in~\eqref{eqt: Vconstraint} and~\eqref{eqt: defSstar}, respectively.
Our goal is to calculate the subdifferential of $S$. To this aim, it suffices to calculate the subdifferential of $\Sstar$,
since $(p,E^-)\in\partial S(x,t)$ holds if and only if $(x,t)\in\partial \Sstar(p,E^-)$ is satisfied. Since $J^*$ is finite-valued, which implies that $J^*$ is continuous in the whole space $\Rn$, then by Proposition~\ref{prop: bkgd_linear_subdiff}, the subdifferential operator is linear with respect to the summation in $\Sstar$. Hence, we have 
\begin{equation} \label{eqt: pf_subdiff_G}
    \partial \Sstar (p,E^-) = \partial J^*(p)\times\{0\} + \partial \ind_\Vconstraint(p,E^-) = \partial J^*(p)\times\{0\} + N_\Vconstraint(p,E^-),
\end{equation}
where $N_\Vconstraint(p,E^-)$ denotes the normal cone of $\Vconstraint$ at $(p,E^-)$, and the last equality follows from~\eqref{eqt:subgrad_indicator}.

Now, we calculate the normal cone of $\Vconstraint$. 
First, if $p\in \interior\dom H$ and $E^-<-H(p)$ hold, then $(p,E^-)$ is in the interior of $\Vconstraint$ where the normal cone is $\{0\}$. 
Then, we consider the case when $p\in\interior\dom H$ and $E^-=-H(p)$. Since the set $\Vconstraint$ is the reflection of the epigraph of the Legendre function $H$,~\cite[Proposition~D.1.3.1]{hiriart2004} implies 
\begin{equation}\label{eqt: prop51cpf_normalcone_formula1}
    N_\Vconstraint(p,E^-) = \{\alpha(\nabla H(p), 1)\colon \alpha \geq 0\}.
\end{equation}
It remains to consider the case when $p$ is a point on the boundary of $\dom H$. Let $p$ be a vector in $(\dom H)\setminus(\interior\dom H)$, and $E^-$ be a scalar satisfying $E^-\leq-H(p)$. By definition of normal cone, $(y,s)\in N_\Vconstraint(p,E^-)$ holds if and only if $(y,s)$ verifies
\begin{equation}\label{eqt: prop51cpf_normalcone}
    0\geq \langle y, \tilde{p} - p\rangle + s(\tilde{E}^- - E^-),\quad \forall\, \tilde{p}\in\dom H, \, \tilde{E}^-\leq - H(\tilde{p}).
\end{equation}
Let $(y,s)$ satisfy~\eqref{eqt: prop51cpf_normalcone}. Choosing $\tilde{p}=p$ and $\tilde{E}^-<E^-$ yield $s\geq 0$. If $s>0$, letting $\tilde{p}$ be an arbitrary point in $\dom H$ and $\tilde{E}^- = -H(\tilde{p})$, and dividing~\eqref{eqt: prop51cpf_normalcone} by $s$, we get
\begin{equation*}
    0\geq \left\langle \frac{y}{s}, \tilde{p} - p\right\rangle  -H(\tilde{p}) - E^-\geq \left\langle \frac{y}{s}, \tilde{p} - p\right\rangle -H(\tilde{p}) +H(p).
\end{equation*}
Since the inequality holds for an arbitrary point $\tilde{p}\in\dom H$, we conclude that $\frac{y}{s}\in\partial H(p)$, which contradicts with $\partial H(p)=\emptyset$ by the assumption that $H$ is Legendre and $p$ is a boundary point. Therefore, we conclude that $s=0$ holds, and then~\eqref{eqt: prop51cpf_normalcone} is equivalent to $y\in N_{\dom H}(p)$. In other words, we have
\begin{equation}\label{eqt: prop51cpf_normalcone_formula2}
    N_\Vconstraint(p,E^-) = N_{\dom H}(p)\times \{0\}.
\end{equation}

So far, we have calculated the normal cone $N_\Vconstraint(p,E^-)$
at different points $(p,E^-)$ in~\eqref{eqt: prop51cpf_normalcone_formula1} and~\eqref{eqt: prop51cpf_normalcone_formula2}, from which we get for any $(p,E^-)\in \Vconstraint$
\begin{equation} \label{eqt: pf_subdiff_ind}
    \partial \ind_\Vconstraint(p,E^-) = \begin{dcases}
    \{(0,0)\}, &\text{if }p\in\interior\dom H,\, E^-<-H(p),\\
    \{\alpha(\nabla H(p), 1)\colon \alpha \geq 0\}, &\text{if }p\in\interior\dom H,\, E^-= -H(p),\\
    N_{\dom H}(p)\times \{0\}, & \text{if }
    p\not\in \interior\dom H,
    \, E^-\leq-H(p).
    \end{dcases}
\end{equation}
By combining~\eqref{eqt: pf_subdiff_G} and~\eqref{eqt: pf_subdiff_ind}, we obtain for any $(p,E^-)\in \Vconstraint$
\begin{equation}\label{eqt: prop51_pf_subdiff_G}
    \partial \Sstar (p,E^-) = \begin{dcases}
    \partial J^*(p)\times \{0\}, &\text{if }p\in\interior\dom H,\, E^-<-H(p),\\
    \{(u+\alpha \nabla H(p), \alpha)\colon \alpha \geq 0, u\in\partial J^*(p)\}, &\text{if }p\in\interior\dom H,\, E^-= -H(p),\\
    \left(\partial J^*(p) + N_{\dom H}(p)\right)\times \{0\}, & \text{if }p\not\in \interior\dom H,\, E^-\leq-H(p).
    \end{dcases}
\end{equation}
Let $(x,t)\in\interior\dom S$, which implies $t>0$.
Since $t>0$ holds, we conclude that $(x,t)\in\partial \Sstar(p,E^-)$ holds if and only if $(x,t)$ satisfies the second line in~\eqref{eqt: prop51_pf_subdiff_G}, i.e., 
there hold $p\in\interior\dom H$, $E^-+H(p)= 0$ and $u+t\nabla H(p)=x$ for some $u\in\partial J^*(p)$. As a result, we have
\begin{equation} \label{eqt: pf_subdiff_S}
\begin{split}
    \partial S(x,t) &= 
    \{(p,E^-)\in\R^n\times \R\colon (x,t)\in\partial \Sstar(p,E^-)\} \\
    &= \{(p,-H(p))\in\R^n\times \R \colon p\in\interior\dom H,
    u+t\nabla H(p)=x \text{ for some } u\in\partial J^*(p)\}\\
    &= \{(p,-H(p))\in\R^n\times \R \colon
    u+t\nabla H(p)=x \text{ for some } u\in\partial J^*(p)\},
\end{split}
\end{equation}
where in the last equality the constraint $p\in\interior\dom H$ is removed since it is automatically satisfied when $H$ is differentiable at $p$ (by definition of Legendre function).
According to (b), the assumption $(x,t)\in\interior\dom S$ implies $x\in \dom J + t\,\interior\dom H^*$, and hence the assumption in Proposition~\ref{proposition: existence_uniqueness} is satisfied. 
Note that the constraint in the set in the last line of~\eqref{eqt: pf_subdiff_S} is $u+t\nabla H(p)=x$ for some $u\in\partial J^*(p)$. Hence, a point $p$ verifies the constraint if and only if there exists $u\in\R^n$ such that $u$ and $p$ satisfy statement (b) in Proposition~\ref{proposition: existence_uniqueness}, which is true if and only if $p$ is the unique maximizer of~\eqref{eqt: lemma_p_opt}.
Therefore, the set $\partial S(x,t)$ is a singleton set containing $(\pbar, -H(\pbar))$, where $\pbar$ is the unique maximizer of~\eqref{eqt: lemma_p_opt}. 
Since $(x,t)$ is in the interior of $\dom S$,
applying~\cite[Corollary~D.2.1.4]{hiriart2004} to a neighborhood of $(x,t)$ contained in $\dom S$, we conclude that
$S$ is differentiable at $(x,t)$, and the gradient is the element in the singleton set $\partial S(x,t)$, which proves~\eqref{eqt: gradS}.

(d) We have proved in (c) that $S$ is differentiable at $(x,t)\in\interior\dom S$, and the gradient satisfies~\eqref{eqt: gradS}. As a result, the first line in~\eqref{eqt: mainresult_HJ_S} is satisfied. Then, it suffices to check the initial condition in~\eqref{eqt: mainresult_HJ_S}. 
By the assumption on $\icHJS$, we conclude that $\icHJS^* = J^*+\ind_{\cl\dom H}$ holds. 
Since $\icHJS$ is in $\gmRn$, we have
\begin{equation*}
    \icHJS(x) = (\icHJS^*)^*(x) =  (J^*+\ind_{\cl \dom H})^*(x) = S(x,0),
\end{equation*}
where the last equality holds by~\eqref{eqt: pf_pdeb_S0}.
Therefore, the initial condition is satisfied, and $S$ solves the HJ PDE~\eqref{eqt: mainresult_HJ_S}.

(e) 
For all $x\in\R^n$, by~\eqref{eqt: pf_pdeb_S0}, we have
\begin{equation}\label{eqt: pf_S0_equals_J}
    S(x,0) = \sup_{p\in\cl \dom H} \{\langle x,p\rangle - J^*(p)\} \leq \sup_{p\in\R^n} \{\langle x,p\rangle - J^*(p)\} = J(x).
\end{equation}
Now, let $x\in\R^n$ satisfy $\partial J(x)\cap \cl\dom H\neq \emptyset$. Let $q$ be a vector in $\partial J(x)\cap \cl\dom H$. 
Since $q\in\partial J(x)$, then $q$ is a maximizer of $p\mapsto \langle x,p\rangle - J^*(p)$. And $q$ is also in $\cl \dom H$, which implies that
\begin{equation*}
    q \in \argmax_{p\in\cl\dom H} \{\langle x,p\rangle - J^*(p)\}.
\end{equation*}
Then, $q$ is a maximizer of the two maximization problems in~\eqref{eqt: pf_S0_equals_J}, and hence the inequality in~\eqref{eqt: pf_S0_equals_J} becomes equality, which implies $S(x,0)=J(x)$. Moreover, by~\eqref{eqt: prop51pf_S0star}, the Legendre-Fenchel transform of $y\mapsto S(y,0)$ is $J^* + \ind_{\cl\dom H}$. As a result, $p\in \partial_x S(x,0)$ is verified if and only if there holds
\begin{equation*}
    \langle x,p\rangle = S(x,0) + J^*(p) + \ind_{\cl\dom H}(p) = J(x) + J^*(p) + \ind_{\cl \dom H}(p),
\end{equation*}
which is true if and only if $p$ is in $\partial J(x)\cap \cl\dom H$. Therefore, we have $\partial_x S(x,0) = \partial J(x) \cap \cl\dom H$. 
\qed
\end{proof}

\begin{remark}
We use the notation $(p,E^-)$ to denote the dual variables due to their physics meaning. In HJ theory, the variables $x$ and $t$ give the location and time of a particle. The solution to the HJ PDE gives the action, which is denoted by $S$. The dual variable of $x$ gives the momentum of the particle, which is usually denoted by $p$. The dual variable of $t$ is the time derivation of the action, which equals the negative energy ($-H(p)$) by the HJ PDE. Hence, we use the notation $E^-$, where $E$ denotes the energy and the minus sign in the superscript denotes the negative sign in the negative energy.
\end{remark}

\begin{remark}
For the sake of  completeness, we calculate the subdifferential of $S$ at $(x,t)$ for all $(x,t)$ in $\dom S$.
The case for $(x,t)\in\interior\dom S$ is calculated in Proposition~\ref{prop: HJPDE}(c).
Then, it remains to consider the case for $(x,t)\in (\dom S)\setminus (\interior \dom S)$.

First, let $t>0$, and hence $x$ is a point in $(\dom J + t\,\dom H^*)\setminus (\dom J + t\,\interior\dom H^*)$ by Proposition~\ref{prop: HJPDE}(b).
Note that~\eqref{eqt: prop51_pf_subdiff_G} and~\eqref{eqt: pf_subdiff_S} still hold since $t$ is not equal to zero. As a result, if $(p,E^-)$ is a point in $\partial S(x,t)$, then there exists $u\in\partial J^*(p)\subseteq \dom J$ satisfying $u+t\nabla H(p) = x$. Let $v=\nabla H(p)$, which implies that $v$ is in the interior of $\dom H^*$ (see Remark~\ref{rem: properties_HandJ}). Hence, we have $x = u+tv\in\dom J + t\,\interior\dom H^*$, which leads to a contradiction. Therefore, we get $\partial S(x,t) = \emptyset$ in this case.

Then, we consider the case when $t=0$. By~\eqref{eqt: prop51_pf_subdiff_G}, $(x,0)\in \partial \Sstar(p,E^-)$ is satisfied if and only if one of the following two cases hold
\begin{itemize}
    \item[(a)] $x\in \partial J^*(p)$, $p\in\interior\dom H$ and $E^-\leq -H(p)$ (which corresponds to the first and second lines in~\eqref{eqt: prop51_pf_subdiff_G}),
    \item[(b)] $x\in\partial J^*(p) + N_{\dom H}(p)$, $p\in(\dom H)\setminus(\interior\dom H)$ and $E^-\leq -H(p)$ (which corresponds to the third line in~\eqref{eqt: prop51_pf_subdiff_G}).
\end{itemize}
Since $N_{\dom H}(p) = \{0\}$ if $p$ is in the interior of $\dom H$, the two cases above can be combined to the conditions that $x\in\partial J^*(p) + N_{\dom H}(p)$, $p\in \dom H$ and $E^-\leq -H(p)$. Note that there holds $\partial J^*(p) + N_{\dom H}(p) = \partial (J^* + \ind_{\cl\dom H})(p)$ for all $p\in \dom H$.
As a result, $(p,E^-)\in\partial S(x,0)$ is verified if and only if 
\begin{equation*}
    p\in (\partial (J^* + \ind_{\cl\dom H})^*(x)) \cap \dom H
    = \partial_x S(x,0)\cap \dom H\text{ and } E^-\leq -H(p),
\end{equation*}
where $\partial_x S(x,0)$ denotes the subdifferential of the map $y\mapsto S(y,0)$. Hence, in this case we have
\begin{equation*}
    \partial S(x,0) = \{(p,E^-)\in(\partial_x S(x,0)\cap \dom H)\times \R\colon E^-\leq -H(p)\}.
\end{equation*}
Therefore, the subdifferential of $S$ at each point $(x,t)$ in $\dom S$ reads
\begin{equation*}
\partial S(x,t) = \begin{dcases}
\{(\pbar, -H(\pbar))\}, & \text{if }(x,t)\in\interior\dom S,\\
\{(p,E^-)\in(\partial_x S(x,0)\cap \dom H)\times \R\colon E^-\leq -H(p)\}, & \text{if }(x,t)\in \dom S,\,t=0\\
\emptyset, & \text{otherwise},
\end{dcases}
\end{equation*}
where $\pbar$ in the first line is the unique maximizer in~\eqref{eqt: lemma_p_opt} at $(x,t)$.
\end{remark}

\begin{remark}
The initial condition in~\eqref{eqt: S_equiv_formula} is the initial condition of~\eqref{eqt: mainresult_HJ_S}, i.e, $\icHJS = (J^*+ \ind_{\dom H})^*$ holds. 
Note that the initial condition in~\eqref{eqt: intro_HJ_S} is $J$ which may be different from $\icHJS$.
Specifically, if $\dom J^*$ is included in $\cl\dom H$, then the initial condition $\icHJS$ in~\eqref{eqt: mainresult_HJ_S} equals the initial condition $J$ in~\eqref{eqt: intro_HJ_S}, i.e., $\icHJS = (J^*+ \ind_{\dom H})^* = J$ is satisfied.
\end{remark}

In the sequel, we provide the proofs of Propositions~\ref{prop: mainresult_HJS} and~\ref{prop: mainresult_HJF}.

\textbf{Proof of Proposition~\ref{prop: mainresult_HJS}.}
By Proposition~\ref{prop: HJPDE}(d), the function $S$ solves the HJ PDE~\eqref{eqt: mainresult_HJ_S}.

Let $t>0$ and $x\in \dom J + t\,\interior\dom H^*$. Let $\ubar$ and $\pbar$ be the optimizers in the optimization problems~\eqref{eqt: lemma_u_opt} and~\eqref{eqt: lemma_p_opt}, whose existence and uniqueness are established by Proposition~\ref{proposition: existence_uniqueness}.
Due to Proposition~\ref{prop: HJPDE}(b), $(x,t)$ is a point in $\interior\dom S$.
By Proposition~\ref{prop: HJPDE}(c), $S$ is differentiable at $(x,t)$, and the spatial gradient $\nabla_x S(x,t)$ satisfies
\begin{equation}\label{eqt: prop33pf_gradS}
    \nabla_x S(x,t) = \pbar.
\end{equation}
Then, by~\eqref{eqt: prop32pf_equiv_var2_argmin}, Proposition~\ref{proposition: existence_uniqueness} and~\eqref{eqt: prop33pf_gradS}, the unique minimizer $\bar{\imageq}$ in~\eqref{eqt: mainresult_var2} verifies
\begin{equation*}
\bar\imageq = \frac{x-\ubar}{t} = \nabla H(\pbar) = \nabla H(\nabla_x S(x,t)),
\end{equation*}
which proves~\eqref{eqt: mainresult_vbar_S}.
\qed

\bigbreak

\textbf{Proof of Proposition~\ref{prop: mainresult_HJF}.}
First, let $(x,t)\in\interior\domF$. We show that $F$ is continuously differentiable and satisfies~\eqref{eqt: mainresult_HJ_F} at $(x,t)$. By definition of $\domF$, we have $t>0$ and $\frac{x}{t}\in\interior \dom H^*$. By assumption that $0\in\dom J$, we have $x\in \dom J + t\, \interior \dom H^*$, and hence $(x,t)$ satisfies the assumption of Proposition~\ref{prop: mainresult_Moreau}.
By Proposition~\ref{prop: mainresult_Moreau} and definition of $F$ and $S$, we obtain
\begin{equation}\label{eqt: pf_HJF_1}
    F(x,t) + S(x,t) = (tH)^*(x) = tH^*\left(\frac{x}{t}\right).
\end{equation}
Since $H^*$ is differentiable at $\frac{x}{t}$ (see Remark~\ref{rem: properties_HandJ}), and $S$ is differentiable at $(x,t)$ by Proposition~\ref{prop: HJPDE}(b)-(c), it follows that $F$ is differentiable at $(x,t)$. Moreover, taking derivative in~\eqref{eqt: pf_HJF_1} and applying the HJ PDE~\eqref{eqt: mainresult_HJ_S}, we get
\begin{equation}\label{eqt: pf_F_grad}
\begin{split}
    \nabla_x F(x,t) &= \nabla H^*\left(\frac{x}{t}\right) - \nabla_x S(x,t),\\
    \frac{\partial F(x,t)}{\partial t} &= H^*\left(\frac{x}{t}\right) - \left\langle \frac{x}{t}, \nabla H^*\left(\frac{x}{t}\right)\right\rangle -\frac{\partial S(x,t)}{\partial t}\\
    &= H^*\left(\frac{x}{t}\right) - \left\langle \frac{x}{t}, \nabla H^*\left(\frac{x}{t}\right)\right\rangle + H(\nabla_x S(x,t))\\
    &= H^*\left(\frac{x}{t}\right) - \left\langle \frac{x}{t}, \nabla H^*\left(\frac{x}{t}\right)\right\rangle + H\left(\nabla H^*\left(\frac{x}{t}\right) - \nabla_x F(x,t)\right).
\end{split}
\end{equation}
Applying Proposition~\ref{prop: bkgd_conjugate_equal_inverse} to the function $H^*$, we get
\begin{equation*}
    \begin{split}
    \frac{\partial F(x,t)}{\partial t} 
    &= -H\left(\nabla H^*\left(\frac{x}{t}\right)\right) + H\left(\nabla H^*\left(\frac{x}{t}\right) - \nabla_x F(x,t)\right).
\end{split}
\end{equation*}
Therefore, the differential equation in~\eqref{eqt: mainresult_HJ_F} is satisfied. 

Now, let $d$ be a vector in $\interior\dom H^*$ and $x$ be a vector in $\dom (\ind_{\dom H}^*)$. We prove that~\eqref{eqt: mainresult_HJ_F_continuity} holds.
By Proposition~\ref{prop: bkgd_asym2}, we have $\ind_{\dom H}^* = (H^*)'_\infty$. As a result, $x$ is a vector in $\dom (H^*)'_\infty$, and hence by definition of the asymptotic function we have $\frac{x}{t}+d\in \dom H^*$ for all $t>0$. Since $d$ is in the interior of $\dom H^*$,~\cite[Lemma~A.2.1.6]{hiriart2004} implies that $\frac{x}{t} + d$ is also in $\interior\dom H^*$ for all $t>0$. Then, by Proposition~\ref{prop: HJPDE}(a)-(b) and the assumption that $0\in\dom J$, the point $(x+td,t)$ is in the interior of $\dom S$. Hence, by Proposition~\ref{prop:bkgd_convex_cont}, $S$ is continuous on the straight line connecting $(x,0)$ and $(x+td,t)$, i.e., there holds
\begin{equation}\label{eqt:pf_prop34_contS}
    \lim_{t\to 0^+} S(x+td,t) = S(x,0).
\end{equation}
Also, by definition~\eqref{eqt:defasym}, we obtain
\begin{equation}\label{eqt:pf_prop34_contHstar}
    \lim_{t\to 0^+} tH^*\left(\frac{x+td}{t}\right)=
    \lim_{t\to 0^+} t\left(H^*\left(d+\frac{x}{t}\right) - H^*(d)\right)
    = (H^*)'_\infty(x).
\end{equation}
Note that~\eqref{eqt: pf_HJF_1} holds at $(x+td,t)$ since we have proved that $\frac{x}{t} + d$ is a point in $\interior\dom H^*$. Then, by straightforward computation using~\eqref{eqt: pf_HJF_1},~\eqref{eqt:pf_prop34_contS},~\eqref{eqt:pf_prop34_contHstar}, Proposition~\ref{prop: bkgd_asym2},~\eqref{eqt: mainresult_defS} and~\eqref{eqt: mainresult_defF}, we have
\begin{equation*}
\begin{split}
    \lim_{t\to 0^+} F(x+td,t)
    &= \lim_{t\to 0^+} tH^*\left(\frac{x+td}{t}\right) - \lim_{t\to 0^+} S(x+td,t)
    = (H^*)'_\infty(x) - S(x,0)\\
    &= \ind_{\dom H}^*(x) - (J^*+\ind_{\dom H})^*(x)
    = F(x,0),
\end{split}
\end{equation*}
which proves~\eqref{eqt: mainresult_HJ_F_continuity}.

Thus, it remains to show~\eqref{eqt: mainresult_barv_F}. Let $t>0$ and $x\in t\,\interior\dom H^*$. Denote by $\bar{\imageq}$ be the unique minimizer in~\eqref{eqt: mainresult_var1}. By Propositions~\ref{prop: mainresult_Moreau} and~\ref{prop: mainresult_HJS}, we have
\begin{equation*}
    \bar{\imageq} = \nabla H(\nabla_x S(x,t)) = \nabla H\left(\nabla H^*\left(\frac{x}{t}\right) - \nabla_x F(x,t)\right),
\end{equation*}
where the second equality holds according to the first equation in~\eqref{eqt: pf_F_grad}.
\qed

\section{Poisson noise model} \label{sec: Poisson}
 The following variational problem is usually employed in the literature to handle Poisson noise (see, e.g., \cite{bertero_2009,poisson_primal_dual,le_chartrand_asaki}),
\begin{equation}\label{eqt: Poisson_q_opt}
    \min_{\imageq\in(0,+\infty)^n} \left\{\sum_{i=1}^n\left(t\imageq_i - x_i\log \imageq_i + x_i\log\left(\frac{x_i}{t}\right) - x_i\right) + 
    f(\imageq)\right\},
\end{equation}
where $f(\imageq)$ is a regularization term. 
Define $H\in\gmRn$ by
\begin{equation}\label{eqt: defH_Poisson}
    H(\pvar) = \sum_{i=1}^n e^{\pvar_i}, \quad \forall \, \pvar=(\pvar_1,\dots, \pvar_n)\in\R^n.
\end{equation}
The function $H$ is a Legendre function, and hence it satisfies assumption (A1). By straightforward computation, its Legendre-Fenchel transform $H^*$ is given by
\begin{equation*}
    H^*(y) = \sup_{\pvar\in\R^n} \sum_{i=1}^n (\pvar_i y_i - e^{\pvar_i}) = 
    \begin{dcases}
    \sum_{i=1}^n (y_i\log y_i - y_i), & \text{if } y\in [0,+\infty)^n,\\
    +\infty, & \text{otherwise}.
    \end{dcases}
\end{equation*}
Hence, the Bregman distance $D_{H^*}(\frac{x}{t}, \imageq)$ reads
\begin{equation*}
\begin{split}
D_{H^*}\left(\frac{x}{t}, \imageq\right)
&= \sum_{i=1}^n \left(\frac{x_i}{t} \log \frac{x_i}{t} - \frac{x_i}{t} - \imageq_i\log \imageq_i + \imageq_i - \left(\frac{x_i}{t} - \imageq_i\right)\log \imageq_i\right)\\
&= \sum_{i=1}^n \left(\frac{x_i}{t} \log \frac{x_i}{t} - \frac{x_i}{t} + \imageq_i - \frac{x_i}{t}\log \imageq_i\right).
\end{split}
\end{equation*}
Therefore, the variational model~\eqref{eqt: Poisson_q_opt}
with Poisson noise can be expressed as 
\begin{equation}\label{eqt:variational_model_Bregman_f}
    \min_{\imageq\in(0,+\infty)^n} \left\{tD_{H^*}\left(\frac{x}{t},\imageq\right) + f(\imageq)\right\}.
\end{equation}
If there exists a 1-coercive function $J\in\gmRn$ satisfying
\begin{equation}\label{eqt:poissonf}
    f(\imageq) = J^*(\nabla H^*(\imageq)) = J^*(\log \imageq_1, \dots, \log \imageq_n),
\end{equation}
for all $\imageq\in(0,+\infty)^n$, then the variational model~\eqref{eqt: Poisson_q_opt} is in the same form as~\eqref{eqt: mainresult_var1} with $H$ defined in~\eqref{eqt: defH_Poisson}. 

By Proposition~\ref{prop: mainresult_Moreau}, this is related to an additive model in~\eqref{eqt: mainresult_var2}, which in this specific example reads
\begin{equation} \label{eqt: Poisson_additive_q}
    \min_{\imageq\in\R^n} \left\{J(x-t\imageq) + t\sum_{i=1}^n (\imageq_i\log \imageq_i - \imageq_i)\right\}.
\end{equation}
Furthermore, by Propositions~\ref{prop: mainresult_HJS} and~\ref{prop: mainresult_HJF}, these variational problems are related to two PDEs~\eqref{eqt: mainresult_HJ_S} and~\eqref{eqt: mainresult_HJ_F}, which,
in this case, become
\begin{equation*}
\frac{\partial S}{\partial t}(x,t) + \sum_{i=1}^n \exp\left(\frac{\partial S}{\partial x_i}(x,t)\right) = 0,
\end{equation*}
and
\begin{equation*}
    \frac{\partial F}{\partial t}(x,t)  - \frac{1}{t}\sum_{i=1}^n x_i\left(\exp\left(-\frac{\partial F}{\partial x_i}(x,t)\right)-1\right)=0.
\end{equation*}

\begin{remark}
An example of penalizing function in~\eqref{eqt:poissonf} is
\begin{equation}\label{eqt:tvlog_prior}
f(v)=\tv(\log v_1,\dots, \log v_n).
\end{equation}
This  can be written as in~\eqref{eqt:poissonf}, where the function $J^*$ is the total variation, while
  $J$  is the indicator ball of Meyer's norm  (see \cite[Definition 10, page 30]{meyer2001oscillating}) and  satisfies assumption (A1). The corresponding variational denoising model~\eqref{eqt: Poisson_q_opt} becomes
\begin{equation*}
    \min_{v\in\R^n\colon v_i>0} \left\{\sum_{i=1}^n\left(tv_i - x_i\log v_i + x_i\log\left(\frac{x_i}{t}\right) - x_i\right) + \tv(\log v_1, \dots, \log v_n)\right\}.
\end{equation*}
This particular regularization  has been dealt with in the literature - see~\cite{poisson_logtv1,poisson_mult_logtv}.
\end{remark}

\begin{remark}\label{rem:poisson_tv_prior}

A widely used regularization term is the total variation $\tv(v)$. However, this cannot be expressed as~\eqref{eqt:poissonf}, since there is no convex lower semi-continuous function $J$ satisfying $J^*(\log v_1, \dots, \log v_n) = \tv(v)$. 
\end{remark}

\begin{remark} \label{rem:meaningt_poisson}
Note that the parameter $t$ in~\eqref{eqt: Poisson_q_opt} (which is the time variable in the corresponding HJ PDEs) is related to the exposure time of the sensor. To be specific, let $v$ be the gray level array of the original image, and assume there is no motion and the gray level image $v$ does not change over time. The observed image is a sample from a Poisson distribution whose rate equals $tv$, where $t$ is the exposure time of the sensor (see~\cite{Tendero2016Coded}). The probability mass function of the Poisson distribution at $x\in\mathbb{Z}^n$ equals
\begin{equation}\label{eqt:stat_poisson}
    P(x|v) = \prod_{i=1}^n\frac{(tv_i)^{x_i}e^{-tv_i}}{x_i!} = \exp\left(-\sum_{i=1}^n (tv_i - x_i\log v_i - x_i\log t + \log(x_i!))\right).
\end{equation}
Then, the corresponding MAP estimator for the denoising problem with Poisson noise reads
\begin{equation*}
\begin{split}
    \bar{v} &= \argmin_{v\in (0,+\infty)^n} \left\{\sum_{i=1}^n\left(tv_i - x_i\log v_i - x_i\log t + \log(x_i!)\right) + f(v)\right\}\\
    &= \argmin_{v\in (0,+\infty)^n} \left\{\sum_{i=1}^n(tv_i - x_i\log v_i) + f(v)\right\},
\end{split}
\end{equation*}
which is equivalent to  model~\eqref{eqt: Poisson_q_opt}.
Therefore, the parameter $t$ in~\eqref{eqt: Poisson_q_opt} is the exposure time of the sensor, and the parameter $x$ in~\eqref{eqt: Poisson_q_opt} is the total number of photons the sensor received during the exposure time.
\end{remark}

\section{Multiplicative noise model} \label{sec: multiplicative}
The following variational problem~\cite{Aubert2008Variational}  has been quite often  employed for denoising problems with multiplicative noise, 
\begin{equation} \label{eqt: multiplicative_q}
    \min_{\imageq\in(0,+\infty)^n} \left\{t\sum_{i=1}^n\left(-1+\log \imageq_i +  \frac{x_i/t}{\imageq_i} -\log\frac{x_i}{t}\right)+ f(\imageq)\right\},
\end{equation}
where $\frac{x}{t}$ is the observed image, $t>0$ is a positive parameter, and $f(\imageq)$ is the regularization term. 
Let $H\in\gmRn$ be defined by
\begin{equation} \label{eqt: defH_multiplicative}
    H(p):= 
    \begin{dcases}
    \sum_{i=1}^n (-1-\log(-p_i)), & \text{if } p = (p_1,\dots, p_n)\in (-\infty, 0)^n,\\
    +\infty, &\text{otherwise.}
    \end{dcases}
\end{equation}
The function $H$ is a Legendre function with domain $(-\infty,0)^n$, hence it satisfies assumption (A1). Its Legendre-Fenchel transform $H^*$ reads
\begin{equation*}
    H^*(y) = \sup_{p\in\R^n} \sum_{i=1}^n (y_ip_i + 1+\log(-p_i))
    = 
    \begin{dcases}
    -\sum_{i=1}^n \log y_i, & \text{if }y=(y_1,\dots, y_n)\in (0,+\infty)^n,\\
    +\infty, &\text{otherwise},
    \end{dcases}
\end{equation*}
which is the Burg entropy.
The Bregman distance $D_{H^*}(\frac{x}{t}, \imageq)$ equals
\begin{equation*}
D_{H^*}\left(\frac{x}{t}, \imageq\right)
= \sum_{i=1}^n\left( -\log\frac{x_i}{t} + \log \imageq_i + \frac{1}{\imageq_i}\left(\frac{x_i}{t} - \imageq_i\right)\right)
= \sum_{i=1}^n\left( -\log\frac{x_i}{t} + \log \imageq_i + \frac{x_i/t}{\imageq_i}-1\right).
\end{equation*}
Note that this Bregman distance $D_{H^*}$ is the Itakura-Saito distance. One can see that the variational model~\eqref{eqt: multiplicative_q} can be written as~\eqref{eqt:variational_model_Bregman_f}, with $H$ defined in~\eqref{eqt: defH_multiplicative}.

If there exists a 1-coercive function $J\in\gmRn$ such that $f(\imageq)$ satisfies
\begin{equation}\label{eqt:prior_multiplicative}
    f(\imageq) = J^*(\nabla H^*(\imageq)) = J^*\left(-\frac{1}{\imageq_1},\dots, -\frac{1}{\imageq_n}\right),\quad\forall\, \imageq\in(0,+\infty)^n,
\end{equation}
then the variational problem~\eqref{eqt: multiplicative_q} is in the form of~\eqref{eqt: mainresult_var1} with $H$ defined in~\eqref{eqt: defH_multiplicative}.

By Proposition~\ref{prop: mainresult_Moreau}, this variational problem is related to the additive model~\eqref{eqt: mainresult_var2}, which in this case reads
\begin{equation*}
    \min_{\imageq\in(0,+\infty)^n} \left\{J(x-t\imageq) - t\sum_{i=1}^n \log\imageq_i\right\}.
\end{equation*}
Furthermore, by Propositions~\ref{prop: mainresult_HJS} and~\ref{prop: mainresult_HJF}, these variational problems are related to two PDEs~\eqref{eqt: mainresult_HJ_S} and~\eqref{eqt: mainresult_HJ_F}, which
in this case become
\begin{equation*}
\frac{\partial S}{\partial t}(x,t) -\sum_{i=1}^n\left(1+\log \left(-\frac{\partial S}{\partial x_i}(x,t)\right) \right)= 0,
\end{equation*}
and
\begin{equation*}
    \frac{\partial F}{\partial t}(x,t)  + \sum_{i=1}^n \log\left(1+\frac{x_i}{t}\frac{\partial F}{\partial x_i}(x,t)\right) =0.
\end{equation*}

\begin{remark} \label{rem:mul_tvlog_prior}
The non-convex regularization term in~\eqref{eqt:tvlog_prior} has been employed by several authors for multiplicative denoising (see~\cite{Shi2008Nonlinear,Jin2010Analysis,huang_2009}). The corresponding variational model reads
\begin{equation}\label{eqt:mul_tvlog}
    \min_{\imageq\in(0,+\infty)^n} \left\{t\sum_{i=1}^n\left(-1+\log \imageq_i +  \frac{x_i/t}{\imageq_i} -\log\frac{x_i}{t}\right)+ \tv(\log\imageq_1,\dots, \log\imageq_n)\right\}.
\end{equation}
Note that  $\tv(\log\imageq_1,\dots, \log\imageq_n)$ cannot be a particular instance  of~\eqref{eqt:prior_multiplicative} for some convex function $J$ in $\gmRn$. Therefore, our model~\eqref{eqt: mainresult_var1} does not cover the variational model~\eqref{eqt:mul_tvlog}.
\end{remark}

\begin{remark} \label{rem:meaningt_multiplicative}
Here we discuss the meaning of the parameter $t$ in~\eqref{eqt: multiplicative_q}, which is the time variable in the corresponding HJ PDEs. 
The statistical model for denoising problems with multiplicative noise is explained in~\cite{Aubert2008Variational}.  The observation $J_i$ on the $i$-th pixel  in the model is the average of $L$ observations $I_1,\dots, I_L$, which are i.i.d. sampled from the exponential distribution with rate $\frac{1}{v_i}$.
Hence, the distribution of $J_i$ is the Gamma distribution with parameters $L$ and $\frac{L}{v_i}$, whose density function is
\begin{equation*}
    f_{J_i}(J_i=z_i | v_i) = \left(\frac{L}{v_i}\right)^L\frac{1}{\Gamma(L)} z_i^{L-1}e^{-Lz_i/v_i}, \quad \forall z_i\in [0,+\infty).
\end{equation*}
Since the pixels $J_i$ are independent from each other, the density function of the whole image $J = (J_1,\dots, J_n)$ equals
\begin{equation}\label{eqt:stat_multiplicative}
\begin{split}
    f_{J}(J=z | v) &= \prod_{i=1}^n \left(\frac{L}{v_i}\right)^L\frac{1}{\Gamma(L)} z_i^{L-1}e^{-Lz_i/v_i}\\
    &= \exp\left(-\sum_{i=1}^n\left(-L\log L + L\log v_i + \log \Gamma(L) - (L-1)\log z_i + \frac{Lz_i}{v_i}\right)\right),
\end{split}
\end{equation}
for all $z\in [0,+\infty)^n$.
Hence, the corresponding MAP estimator for the denoising problem with multiplicative noise reads
\begin{equation*}
\begin{split}
    \bar{v} &= \argmin_{v\in(0,+\infty)^n}\left\{\sum_{i=1}^n\left(-L\log L + L\log v_i + \log \Gamma(L) - (L-1)\log z_i + \frac{Lz_i}{v_i}\right) + f(v)\right\}\\
    &= \argmin_{v\in(0,+\infty)^n}\left\{\sum_{i=1}^n\left(L\log v_i  + \frac{Lz_i}{v_i}\right) + f(v)\right\},
\end{split}
\end{equation*}
which is equivalent to~\eqref{eqt: multiplicative_q} with $t=L$ and $x= Lz$. 

In other words, the meaning of the time variable $t$ is the number of the observed images, and the meaning of the spatial variable $x$ is the summation of the $t$ observed images.
\end{remark}

\section{An asymptotic result for the variational models}\label{sec:asymptotic}
In this section, we provide an asymptotic result for the variational problems in~\eqref{eqt: prop31_equal_minimizer}. We prove a convergence result of the minimizers under the following assumption.
\begin{itemize}
    \item[(A2)] Let $\vz$ be a vector in $\interior\dom H^*$ and
    $\{\tk\}$ be a sequence of positive numbers going to $+\infty$. Let $\{\xk\}$ be a sequence in $\Rn$ such that $\displaystyle\left\{\frac{\xk}{\tk}\right\}$ converges to $\vz$ as $k$ goes to infinity. Assume $\xk - \tk\vz$ is a vector in $\dom J$ for all $k\in\N$. Further assume that $\displaystyle{\left\{\frac{J(\xk-\tk \vz)}{\tk}\right\}}$ is bounded from above.
\end{itemize}

In the remainder of this paper, the notation $\xk$ stands for the $k$-th vector in a sequence, to avoid ambiguity with $k$-th power of a number $x$ (denoted by $x^k$) and the $k$-th component of a vector $x$ (denoted by $x_k$).
Note that the last statement in (A2) is a technical assumption which is for instance satisfied if $J$ is bounded in its domain (e.g., when $J$ is the indicator function of a closed convex set).
Under this setup, the following proposition proves that the sequence of the minimizers of the two models in~\eqref{eqt: prop31_equal_minimizer} converges to $\vz$ as $k$ goes to infinity. 

\begin{proposition} \label{prop:asymp_tinf}
Assume (A1) and (A2) hold. 
Let $\vk\in\Rn$ be the minimizer in~\eqref{eqt: prop31_equal_minimizer} at $(\xk, \tk)$. Then there is $N\in\N$ such that $\vk$ exists and is unique for all $k\geq N$. Moreover, there holds $\lim_{k\to\infty} \vk = \vz$.
\end{proposition}
\begin{proof}
Let $\dk  := \xk - \tk\vz$. 
By assumption (A2), we have $\dk \in \dom J$ for all $k\in\N$, $\{\frac{\dk }{\tk}\}$ converges to zero as $k$ goes to infinity, and $\{\frac{J(\dk )}{\tk}\}$ is bounded from above.
Since the vector $\dk $ is in $\dom J$, and the vector $\vz$ is in $\interior \dom H^*$, we have $\xk\in \dom J + \tk\interior\dom H^*$ for all $k\in\N$. 
Moreover, since $\{\frac{\xk}{\tk}\}$ converges to $\vz\in \interior\dom H^*$, there exists $N\in\N$ such that $\xk$ is in the set $\tk\dom H^*$ for $k\geq N$.
Therefore, the assumptions of Proposition~\ref{prop: mainresult_Moreau} hold at $(\xk,\tk)$ for all $k\geq N$, and hence the minimizer $\vk$ exists and is unique for all $k\geq N$.

Now, we prove the convergence of $\{\vk\}$ to $\vz$ by contradiction. Assume that the sequence $\{\vk\}$ does not converge to $\vz$. Then, there exists a subsequence, still denoted by $\{\vk\}$, which satisfies $\|\vk-\vz\|\geq \epsilon$ for all $k\in\N$ for some positive constant $\epsilon$. Hence, we have
\begin{equation*}
\begin{split}
    \liminf_{k\to\infty} \|\xk-\tk\vk\| &= \liminf_{k\to\infty}\|\dk  + \tk(\vz - \vk)\|
    = \liminf_{k\to\infty}\tk\left\|\frac{\dk }{\tk} + \vz - \vk\right\|\\
    &\geq 
    \liminf_{k\to\infty}\tk\frac{\epsilon}{2} = +\infty,
\end{split}
\end{equation*}
where the inequality holds since $\frac{\dk }{\tk}$ converges to zero by assumption (A2) and $\|\vk-\vz\|\geq \epsilon$ holds.
Recall that $J$ is 1-coercive, and hence for all $M>0$, there exists $K\in\mathbb{N}$ such that there holds
\begin{equation*}
    J(\xk-\tk\vk) \geq M\|\xk-\tk\vk\|, \quad \forall\,k\geq K.
\end{equation*}
Since $H^*$ is a convex lower semi-continuous function, there exist $p\in\Rn$ and $a\in\R$ satisfying $H^*(x)\geq \langle p,x\rangle +a$ for all $x\in\Rn$.
Therefore, for all $k\geq K$, we have
\begin{equation*}
\begin{split}
    \frac{S(\xk,\tk)}{\tk} &= H^*(\vk) + \frac{1}{\tk}J(\xk-\tk\vk) \geq \langle p,\vk\rangle + a + \frac{M\|\xk-\tk\vk\|}{\tk} \\
    &= \langle p, \vk-\vz\rangle + a +\langle p,\vz\rangle + M\left\|\frac{\dk }{\tk}+ \vz-\vk\right\|\\
    &\geq - \|p\|\|\vk-\vz\| + a +\langle p,\vz\rangle + M\|\vz-\vk\| - M\frac{\|\dk \|}{\tk}\\
    &= (M-\|p\|) \|\vz-\vk\| +a+\langle p,\vz\rangle - M\frac{\|\dk \|}{\tk}\\
    &\geq (M-\|p\|) \epsilon +a+\langle p,\vz\rangle - M\frac{\|\dk \|}{\tk},
\end{split}
\end{equation*}
where the first equality holds by definition of $\vk$ (recall that $S(\xk,\tk)$ equals the optimal value in~\eqref{eqt: prop31_equal_minimizer}), and we assume $M > \|p\|$ in the last inequality.
As $k$ goes to infinity, we have
\begin{equation}\label{eqt:prop91pf_ineqt_geq}
    \liminf_{k\to\infty}\frac{S(\xk,\tk)}{\tk} \geq (M-\|p\|) \epsilon +a+\langle p,\vz\rangle - \limsup_{k\to\infty} M\frac{\|\dk \|}{\tk} = (M-\|p\|) \epsilon +a+\langle p,\vz\rangle,
\end{equation}
where the equality holds since $\{\frac{\dk }{\tk}\}$ converges to zero by assumption (A2).
Moreover, by straightforward computation, we have
\begin{equation*}
    \frac{S(\xk,\tk)}{\tk} = \min_{v\in\Rn} \left\{H^*(v) + \frac{1}{\tk}J(\xk-\tk v)\right\} \leq H^*(\vz) + \frac{1}{\tk}J(\dk ),
\end{equation*}
where the inequality holds by taking $v=\vz$ in the minimization problem. Therefore, we get
\begin{equation}\label{eqt:prop91pf_ineqt_leq}
    \limsup_{k\to\infty} \frac{S(\xk,\tk)}{\tk}\leq 
    \limsup_{k\to\infty}\left\{H^*(\vz) + \frac{1}{\tk}J(\dk )\right\} \leq H^*(\vz) + C,
\end{equation}
 where the last inequality holds since $\{\frac{J(\dk )}{\tk}\}$ is bounded from above by some constant $C$ due to assumption (A2).
Combining~\eqref{eqt:prop91pf_ineqt_geq} and~\eqref{eqt:prop91pf_ineqt_leq}, we obtain
\begin{equation*}
    (M-\|p\|) \epsilon +a+\langle p,\vz\rangle\leq \liminf_{k\to\infty}\frac{S(\xk,\tk)}{\tk} \leq \limsup_{k\to\infty}\frac{S(\xk,\tk)}{\tk} \leq H^*(\vz) + C.
\end{equation*}
Note that the right hand side is a fixed number in $\R$. However, the left hand side can be arbitrarily large, since $M$ is an arbitrary positive number satisfying $M>\|p\|$, and $\epsilon$ is positive. This leads to a contradiction. Therefore, we conclude that $\vk$ converges to $\vz$ as $k$ goes to infinity.
\qed
\end{proof}

\bigbreak

The minimizer of the variational models in~\eqref{eqt: prop31_equal_minimizer} is an 
estimator for some unknown quantity in the corresponding 
imaging
model. Thus, for denoising problems involving Poisson noise, we estimate $v=(v_1,\dots,v_n)$ in the rate parameter $tv$ in the Poisson distribution~\eqref{eqt:stat_poisson}. For multiplicative noise, this estimator is used to evaluate the unknown quantity $v=(v_1,\dots, v_n)$ in the rate parameter $(\frac{L}{v_1}, \dots, \frac{L}{v_n})$ in the Gamma distribution~\eqref{eqt:stat_multiplicative}. From this perspective, Proposition~\ref{prop:asymp_tinf} provides a
convergence result of  the estimator $\vk$ to the vector $\vz$
under assumptions (A1) and (A2). 
Now, we will explain why the limit $\vz$ in Proposition~\ref{prop:asymp_tinf} equals the parameter $v$ in the case of Poisson noise or multiplicative noise.

In the Poisson noise situation, $\tk$ is the exposure time of the sensor, and $\xk$ is the total number of photons the sensor received in this time period (see Remark~\ref{rem:meaningt_poisson}). If $\tk$ is an integer, then the $i$-th component of $\xk$, denoted by $\xk_i$, can be regarded as the sum of $\tk$ independent samples from the Poisson distribution with rate $v_i$. According to the law of large numbers, for all $i\in\{1,\dots,n\}$, as $\tk$ goes to infinity, $\frac{\xk_i}{\tk}$ converges to the expectation of the Poisson distribution with rate $v_i$, which equals $v_i$. In other words, the limit $\vz$ in assumption (A2) equals the parameter $v$ in the 
imaging model. Therefore, in this case, Proposition~\ref{prop:asymp_tinf} shows the convergence of the estimator $\vk$ to the desired quantity $\vz = v$ under assumptions (A1) and (A2).

In the case of multiplicative noise, $\tk$ is the number of images, and $\xk_i$ is the sum of i.i.d. samples from the exponential distribution with rate $\frac{1}{v_i}$ for each $i\in\{1,\dots, n\}$ (see Remark~\ref{rem:meaningt_multiplicative}). 
According to the law of large numbers, for each $i\in\{1,\dots, n\}$, as $\tk$ goes to infinity, $\frac{\xk_i}{\tk}$ converges to the expectation of the exponential distribution with rate $\frac{1}{v_i}$, which equals $v_i$.
Hence, the limit $\vz$ in assumption (A2) equals the parameter $v$ in the imaging model. 
Therefore, in this case, Proposition~\ref{prop:asymp_tinf} shows the convergence of the estimator $\vk$ to the desired quantity $\vz = v$ under assumptions (A1) and (A2).
\begin{remark}
From the proof of Proposition~\ref{prop:asymp_tinf}, we can obtain the value of the asymptotic function $S'_\infty(\vz,1)$ for all $\vz\in \interior\dom H^*$. Let $\xk-\tk\vz$ be a fixed vector $d\in\dom J$. Then,~\eqref{eqt:prop91pf_ineqt_leq} becomes
\begin{equation}\label{eqt:rem82_1}
    \limsup_{k\to\infty} \frac{S(\xk,\tk)}{\tk}\leq 
    \limsup_{k\to\infty}\left\{H^*(\vz) + \frac{1}{\tk}J(d)\right\} = H^*(\vz).
\end{equation}
Moreover, for all $p\in \dom H$, we have
\begin{equation}\label{eqt:rem82_2}
    \liminf_{k\to\infty} \frac{S(\xk,\tk)}{\tk}
    \geq \liminf_{k\to\infty} \frac{\langle \xk,p\rangle - J^*(p) - \tk H(p)}{\tk}
    = \langle \vz,p\rangle - H(p),
\end{equation}
where the inequality holds by~\eqref{eqt: lem_equality_optimal_value} and~\eqref{eqt: S_equiv_formula}, and the equality holds since $\{\frac{\xk}{\tk}\}$ converges to $\vz$ and $J^*(p)$ is finite for all $p\in\Rn$. Taking the supremum over all possible $p\in \dom H$ in~\eqref{eqt:rem82_2} and combining with~\eqref{eqt:rem82_1}, we obtain
\begin{equation*}
    H^*(\vz) = \sup_{p\in\dom H}\{\langle \vz,p\rangle - H(p)\}\leq \liminf_{k\to\infty} \frac{S(\xk,\tk)}{\tk}\leq \limsup_{k\to\infty} \frac{S(\xk,\tk)}{\tk}\leq H^*(\vz),
\end{equation*}
and hence all the inequalities become equalities. Therefore, we have
\begin{equation*}
    S'_\infty(\vz,1) = \lim_{t\to+\infty} \frac{S(d+t\vz,t)-S(d,0)}{t} = \lim_{t\to+\infty} \frac{S(d+t\vz,t)}{t} = H^*(\vz).
\end{equation*}

When $0\in\dom J$, we can obtain the value of the asymptotic function $F'_\infty(\vz,1)$ using~\eqref{eqt: prop31_minimal_value} and~\eqref{eqt: mainresult_defF}. Let $\xk=\tk\vz$ and $\vz\in\interior \dom H^*$ hold. By straightforward calculation, we get
\begin{equation*}
\begin{split}
    F'_\infty(\vz,1) &= \lim_{t\to+\infty} \frac{F(t\vz,t)-F(0,0)}{t} = \lim_{t\to+\infty} \frac{F(t\vz,t)}{t}
    = \lim_{t\to+\infty} \frac{(tH)^*(t\vz) - S(t\vz,t)}{t}\\
    &= H^*(\vz) - S'_\infty(\vz,1) = 0.
\end{split}
\end{equation*}
\end{remark}

\section{Numerical experiments}\label{sec:numerics}
We show some numerical results of the variational models for denoising problems with Poisson noise in section~\ref{sec:numerical_poisson} and multiplicative noise in section~\ref{sec:numerical_multiplicative}, respectively.
For each noise type, we solve  the variational models with the same data fidelity, but with different regularization terms: the one in our model~\eqref{eqt: mainresult_var1} and another one widely  used in the literature. The matlab codes are provided in \url{https://github.com/TingweiMeng/HJ_nonadditive_denoise}. To solve the non-convex problem~\eqref{eqt: mainresult_var1}, we apply the Alternating Direction Method of Multipliers (ADMM) method to the equivalent convex optimization problem~\eqref{eqt: mainresult_var2}, whose $k$-th step reads
\begin{equation}\label{eqt:numerical_admm}
\begin{split}
    \vkp &= \argmin_{v\in \Rn} \left\{ t H^*(v) + \frac{\lambda}{2} \left\| \wk+ tv -x+ \yk\right\|^2\right\},\\
     \wkp &= \argmin_{w\in \Rn} \left\{ J(w) + \frac{\lambda}{2} \|w + t\vkp - x + \yk\|^2\right\},\\
     \ykp &= \yk + \wkp + t\vkp - x.
\end{split}
\end{equation}

\subsection{Poisson noise} \label{sec:numerical_poisson}
We approach the Poisson denoising problem by solving 
\begin{equation} \label{eqt:num_poisson}
\min_{\imageq\in(0,+\infty)^n} \left\{\sum_{i=1}^n\left(t\imageq_i - x_i\log \imageq_i \right) + \alpha \tv(\log\imageq_1,\dots, \log\imageq_n)\right\},
\end{equation}
where $t$ is the positive parameter which denotes the exposure time of the sensor (see Remark~\ref{rem:meaningt_poisson}), $\frac{x}{t}$ is the input noisy image, 
$\alpha$ is a positive parameter in the model, and $\tv(\cdot)$ is the anisotropic total variation. 
This model is a particular instance  of~\eqref{eqt: mainresult_var1} (see section~\ref{sec: Poisson}),
where the function $J$ is the Legendre-Fenchel transform of $\alpha \tv$, which equals the indicator function of the Meyer ball with radius $\alpha$. 
Hence, assumption (A1) is satisfied.
By Proposition~\ref{prop: mainresult_Moreau},  model~\eqref{eqt:num_poisson} is equivalent to the following convex optimization problem
\begin{equation}\label{eqt:num_opt2_poisson}
    \min_{\imageq\in(0,+\infty)^n} \left\{J(x-t\imageq) + t\sum_{i=1}^n (\imageq_i\log \imageq_i - \imageq_i)\right\} = \min_{\imageq\in(0,+\infty)^n} \left\{\tv^*\left(\frac{x-t\imageq}{\alpha}\right) + t\sum_{i=1}^n (\imageq_i\log \imageq_i - \imageq_i)\right\}.
\end{equation}
We apply the ADMM method to solve~\eqref{eqt:num_opt2_poisson} numerically. The update scheme in the $k$-th iteration of ADMM is given as follows
\begin{equation}\label{eqt:admm2_poisson}
\begin{split}
    \vkp &= \argmin_{\imageq\in(0,+\infty)^n} \left\{ t\sum_{i=1}^n (\imageq_i\log \imageq_i - \imageq_i) + \frac{\lambda}{2} \left\| \wk+ tv -x+ \yk\right\|^2\right\},\\
     \wkp &= \argmin_{w\in \Rn} \left\{ J(w) + \frac{\lambda}{2} \|w + t\vkp - x + \yk\|^2\right\}\\
     &= -t\vkp + x - \yk - \argmin_{z\in \Rn} \left\{ \alpha \tv(z) + \frac{1}{2} \|z + t\vkp - x + \yk\|^2\right\},\\
     \ykp &= \yk + \wkp + t\vkp - x.
\end{split}
\end{equation}
The first line in~\eqref{eqt:admm2_poisson} is equivalent to finding $\vkp = (\vkp_1, \dots, \vkp_n)\in {(0,+\infty)^n}$, where each $\vkp_i$ satisfies
\begin{equation}
    t\log \vkp_i + \lambda t^2 \vkp_i = \lambda t(-\wk_i + x_i - \yk_i),
\end{equation}
that is numerically solved using Newton's method.
The second line in~\eqref{eqt:admm2_poisson} involves the proximal map of the anisotropic total variation, which is solved using the algorithm in~\cite{chambolle.09.ijcv,darbon2006imageI,hochbaum.01.jacm}.

We compare the numerical result of~\eqref{eqt:num_poisson} with the following widely used variational model for Poisson denoising problems~\cite{bertero_2009,poisson_primal_dual,le_chartrand_asaki}
\begin{equation}\label{eqt:num_literature_poisson}
    \min_{\imageq\in(0,+\infty)^n} \sum_{i=1}^n\left(t\imageq_i - x_i\log \imageq_i \right) + \alpha \tv(\imageq),
\end{equation}
where the meaning of $x,t,\alpha$ and $\tv$ are the same as in~\eqref{eqt:num_poisson}.
Recall that this model cannot be expressed in the form of~\eqref{eqt: mainresult_var1}, cf. Remark~\ref{rem:poisson_tv_prior}.
We apply ADMM to solve~\eqref{eqt:num_literature_poisson}, where the $k$-th step reads
\begin{equation}\label{eqt:admm_literature_poisson}
\begin{split}
    \vkp &= \argmin_{\imageq\in(0,+\infty)^n} \left\{ \sum_{i=1}^n \left(t\imageq_i - x_i\log \imageq_i \right) + \frac{\lambda}{2} \left\| v - \wk + \yk\right\|^2\right\},\\
     \wkp &= \argmin_{w\in \Rn} \left\{ \alpha \tv(w) + \frac{\lambda}{2} \|\vkp - w + \yk\|^2\right\},\\
     \ykp &= \yk + \vkp - \wkp.
\end{split}
\end{equation}
The first order optimality condition of the optimization problem in the first line of~\eqref{eqt:admm_literature_poisson} gives an analytical formula for $\vkp$, whose $i$-th component reads
\begin{equation*}
    \vkp_i = s_i + \sqrt{s_i^2 + \frac{x_i}{\lambda}}, \quad \text{ where } \quad s_i = \frac{1}{2}\left(\wk_i - \yk_i - \frac{t}{\lambda}\right).
\end{equation*}
The second line in~\eqref{eqt:admm_literature_poisson} is also solved using the algorithm in~\cite{chambolle.09.ijcv,darbon2006imageI,hochbaum.01.jacm}.

\begin{figure}[htbp]
    \centering
    \begin{subfigure}{0.32\textwidth}
        \centering \includegraphics[width=0.9\textwidth]{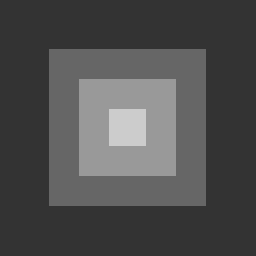}
        \caption{}
    \end{subfigure}
    \begin{subfigure}{0.32\textwidth}
        \centering \includegraphics[width=0.88\textwidth]{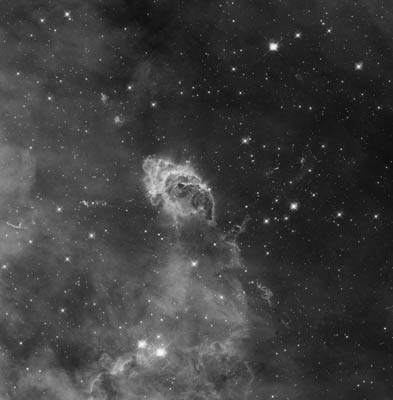}
        \caption{}
    \end{subfigure}
    \begin{subfigure}{0.32\textwidth}
        \centering \includegraphics[width=0.9\textwidth]{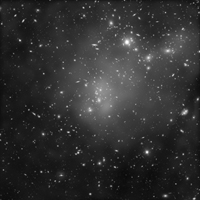}
        \caption{}
    \end{subfigure}
    \caption{The original images in the three testing examples. The middle image is provided by NASA, ESA, and the Hubble SM4 ERO Team in \url{https://hubblesite.org/image/2622/gallery/3-nebulas}.
    The right image is provided by NASA, ESA, J. Merten (Institute for Theoretical Astrophysics, Heidelberg/Astronomical Observatory of Bologna), and D. Coe (STScI) in \url{https://hubblesite.org/contents/media/images/2011/17/2856-Image.html}.}
    \label{fig:ori_images}
\end{figure}

\begin{figure}[htbp]
    \centering
    \begin{subfigure}{0.32\textwidth}
        \centering \includegraphics[width=0.9\textwidth]{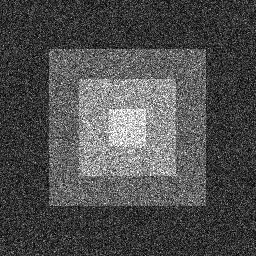}
        \caption{}
    \end{subfigure}
    \begin{subfigure}{0.32\textwidth}
        \centering \includegraphics[width=0.88\textwidth]{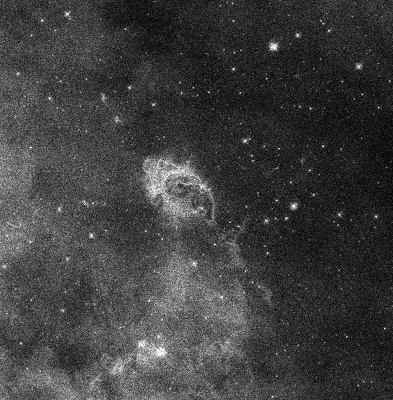}
        \caption{}
    \end{subfigure}
    \begin{subfigure}{0.32\textwidth}
        \centering \includegraphics[width=0.9\textwidth]{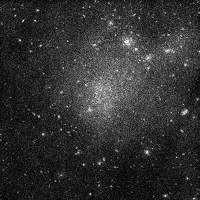}
        \caption{}
    \end{subfigure}
    \caption{The noisy images in the three testing examples corrupted by Poisson noise.}
    \label{fig:noisy_images_poisson}
\end{figure}

We compare the numerical results of the two models~\eqref{eqt:num_poisson} and~\eqref{eqt:num_literature_poisson} on the three test images shown in Figure~\ref{fig:ori_images}, each of them being corrupted by Poisson noise. The noisy image is generated by Poisson distribution as described in Remark~\ref{rem:meaningt_poisson}. The generated noisy images are shown in Figure~\ref{fig:noisy_images_poisson}. To have a fair comparison, we tune the parameters $\alpha$ in different models such that the residual images $\frac{x}{t}-\bar{v}$  have similar $\ell^2$-norms, where $\bar{v}$ is the minimizer in the variational problems.
We also plot the difference between the observed image and the reconstructed image  as proposed in~\cite{Buades2010Image} as the ``method noise".

\begin{figure}[htbp]
    \centering
    \begin{subfigure}{0.49\textwidth}
        \centering \includegraphics[width=0.7\textwidth]{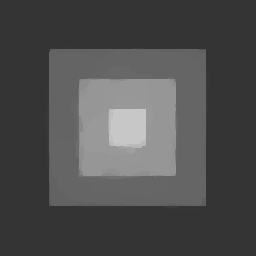}
        \caption{The restored image using~\eqref{eqt:num_poisson}}
    \end{subfigure}
    \begin{subfigure}{0.49\textwidth}
        \centering \includegraphics[width=0.7\textwidth]{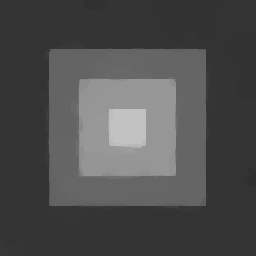}
        \caption{The restored image using~\eqref{eqt:num_literature_poisson}}
    \end{subfigure}\\
    \begin{subfigure}{0.49\textwidth}
        \centering \includegraphics[width=0.7\textwidth]{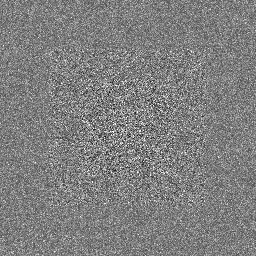}
        \caption{The residual image (+0.5) using~\eqref{eqt:num_poisson}}
    \end{subfigure}
    \begin{subfigure}{0.49\textwidth}
        \centering \includegraphics[width=0.7\textwidth]{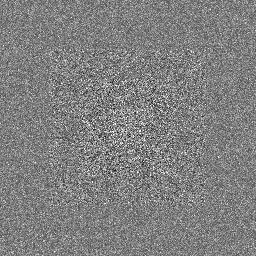}
        \caption{The residual image (+0.5) using~\eqref{eqt:num_literature_poisson}}
    \end{subfigure}
    \caption{The four figures show the restored and residual images of the model~\eqref{eqt:num_poisson} with $\alpha = 6.0$, $t=20$ and~\eqref{eqt:num_literature_poisson} with $\alpha= 12.5$, $t=20$. 
The parameters $\alpha$ are chosen such that the $\ell^2$-norms of the residual images of the two models are close to each other (both are around $31.3$).
\label{fig:poisson_eg1_images}}
\end{figure}

\begin{figure}[htbp]
    \centering
    \begin{subfigure}{0.49\textwidth}
        \centering \includegraphics[width=0.7\textwidth]{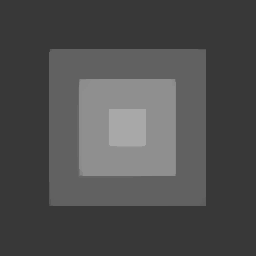}
        \caption{The restored image using~\eqref{eqt:num_poisson}}
    \end{subfigure}
    \begin{subfigure}{0.49\textwidth}
        \centering \includegraphics[width=0.7\textwidth]{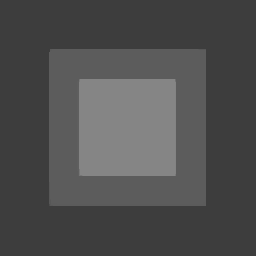}
        \caption{The restored image using~\eqref{eqt:num_literature_poisson}}
    \end{subfigure}\\
    \begin{subfigure}{0.49\textwidth}
        \centering \includegraphics[width=0.7\textwidth]{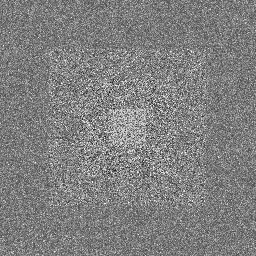}
        \caption{The residual image (+0.5) using~\eqref{eqt:num_poisson}}
    \end{subfigure}
    \begin{subfigure}{0.49\textwidth}
        \centering \includegraphics[width=0.7\textwidth]{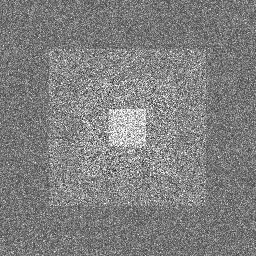}
        \caption{The residual image (+0.5) using~\eqref{eqt:num_literature_poisson}}
    \end{subfigure}
    \caption{The four figures show the restored and residual images of the model~\eqref{eqt:num_poisson} with $\alpha = 6.0$, $t=5$ and~\eqref{eqt:num_literature_poisson} with $\alpha= 12.5$, $t=5$.
\label{fig:poisson_eg1_images_t5}}
\end{figure}

\begin{figure}[htbp]
    \centering
    \begin{subfigure}{0.49\textwidth}
        \centering \includegraphics[width=0.7\textwidth]{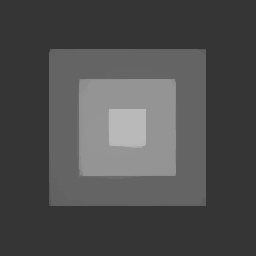}
        \caption{The restored image using~\eqref{eqt:num_poisson}}
    \end{subfigure}
    \begin{subfigure}{0.49\textwidth}
        \centering \includegraphics[width=0.7\textwidth]{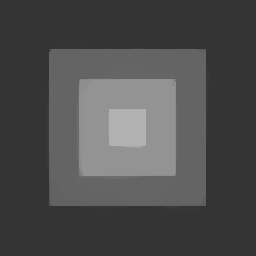}
        \caption{The restored image using~\eqref{eqt:num_literature_poisson}}
    \end{subfigure}\\
    \begin{subfigure}{0.49\textwidth}
        \centering \includegraphics[width=0.7\textwidth]{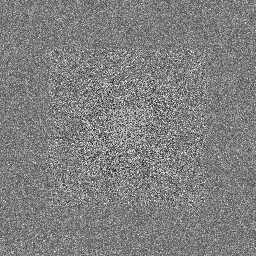}
        \caption{The residual image (+0.5) using~\eqref{eqt:num_poisson}}
    \end{subfigure}
    \begin{subfigure}{0.49\textwidth}
        \centering \includegraphics[width=0.7\textwidth]{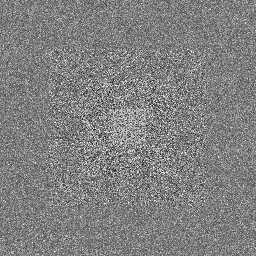}
        \caption{The residual image (+0.5) using~\eqref{eqt:num_literature_poisson}}
    \end{subfigure}
    \caption{The four figures show the restored and residual images of the model~\eqref{eqt:num_poisson} with $\alpha = 6.0$, $t=10$ and~\eqref{eqt:num_literature_poisson} with $\alpha= 12.5$, $t=10$.
\label{fig:poisson_eg1_images_t10}}
\end{figure}

\begin{figure}[htbp]
    \centering
    \begin{subfigure}{0.49\textwidth}
        \centering \includegraphics[width=0.7\textwidth]{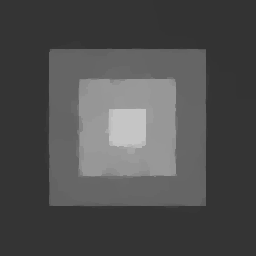}
        \caption{The restored image using~\eqref{eqt:num_poisson}}
    \end{subfigure}
    \begin{subfigure}{0.49\textwidth}
        \centering \includegraphics[width=0.7\textwidth]{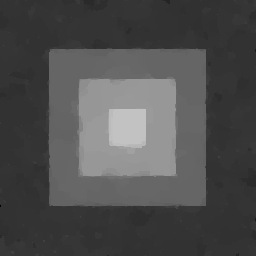}
        \caption{The restored image using~\eqref{eqt:num_literature_poisson}}
    \end{subfigure}\\
    \begin{subfigure}{0.49\textwidth}
        \centering \includegraphics[width=0.7\textwidth]{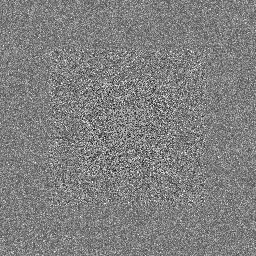}
        \caption{The residual image (+0.5) using~\eqref{eqt:num_poisson}}
    \end{subfigure}
    \begin{subfigure}{0.49\textwidth}
        \centering \includegraphics[width=0.7\textwidth]{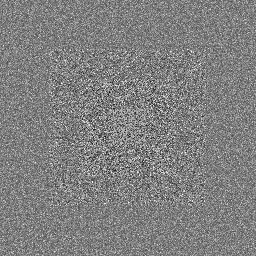}
        \caption{The residual image (+0.5) using~\eqref{eqt:num_literature_poisson}}
    \end{subfigure}
    \caption{The four figures show the restored and residual images of the model~\eqref{eqt:num_poisson} with $\alpha = 6.0$, $t=30$ and~\eqref{eqt:num_literature_poisson} with $\alpha= 12.5$, $t=30$.
\label{fig:poisson_eg1_images_t30}}
\end{figure}

First, we show the numerical results with the input image $\frac{x}{t}$ to be the left image in Figure~\ref{fig:noisy_images_poisson}, where
$t=20.0$. The parameter $\alpha$ in~\eqref{eqt:num_poisson} is set to be $6.0$, while the parameter $\alpha$ in~\eqref{eqt:num_literature_poisson} is  $12.5$. The restored images $\bar{v}$ and the residual images $\frac{x}{t}-\bar{v}$ of these two models are shown in Figure~\ref{fig:poisson_eg1_images}.
We observe similar output images from the two models in this example. Actually both models perform well in the sense of removing noise and preserving edges.
To show the influence of the parameter $t$, we fix the parameter $\alpha$ in both models and show the output images with $t=5$, $t=10$, and $t=30$ in Figures~\ref{fig:poisson_eg1_images_t5},~\ref{fig:poisson_eg1_images_t10}, and~\ref{fig:poisson_eg1_images_t30}, respectively. From these numerical results, we conclude that the parameter $t$ behaves like a regularization parameter (the smaller $t$ is, the stronger the regularization is).

\begin{figure}[htbp]
    \centering
    \begin{subfigure}{0.49\textwidth}
        \centering \includegraphics[width=0.7\textwidth]{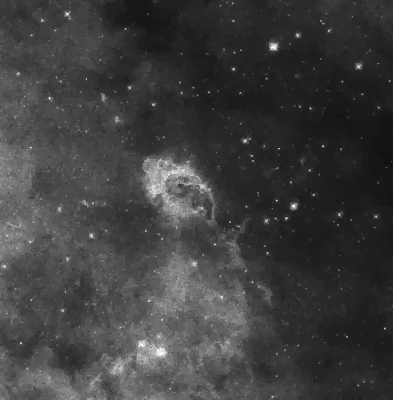}
        \caption{The restored image using~\eqref{eqt:num_poisson}}
    \end{subfigure}
    \begin{subfigure}{0.49\textwidth}
        \centering \includegraphics[width=0.7\textwidth]{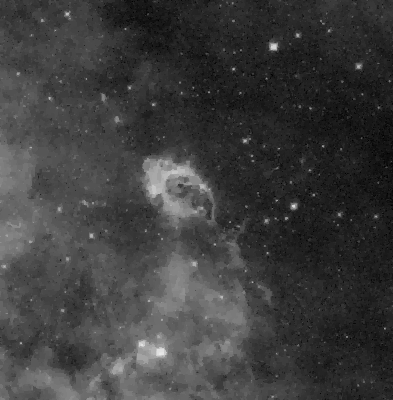}
        \caption{The restored image using~\eqref{eqt:num_literature_poisson}}
    \end{subfigure}\\
    \begin{subfigure}{0.49\textwidth}
        \centering \includegraphics[width=0.7\textwidth]{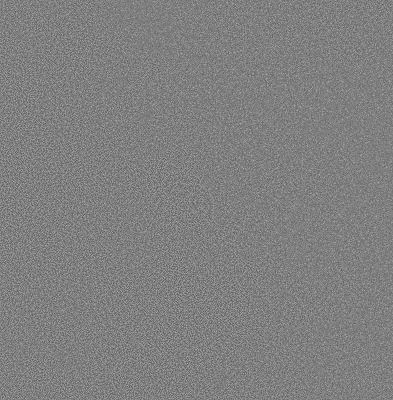}
        \caption{The residual image (+0.5) using~\eqref{eqt:num_poisson}}
    \end{subfigure}
    \begin{subfigure}{0.49\textwidth}
        \centering \includegraphics[width=0.7\textwidth]{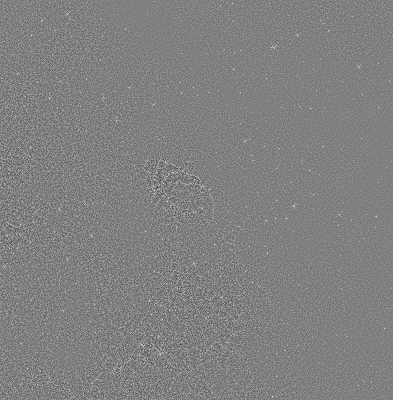}
        \caption{The residual image (+0.5) using~\eqref{eqt:num_literature_poisson}}
    \end{subfigure}
    \caption{The four figures show the restored and residual images of the model~\eqref{eqt:num_poisson} with $\alpha=2.0$, $t=60$ and the model~\eqref{eqt:num_literature_poisson} with $\alpha=7.21875$, $t=60$. 
    The parameters $\alpha$ are chosen such that the $\ell^2$-norms of the residual images of the two models are close to each other (both are around $21.2$).
\label{fig:poisson_eg2_images}}
\end{figure}

Then, we focus on the observed image in Figure~\ref{fig:noisy_images_poisson}(b).
The positive parameter $t$ is $60.0$,  $\alpha$ in~\eqref{eqt:num_poisson} is set to be $2.0$, and the parameter $\alpha$ in~\eqref{eqt:num_literature_poisson} is  $7.21875$. 
 Figure~\ref{fig:poisson_eg2_images}  shows less staircase effect in the restored image of our model~\eqref{eqt:num_poisson} than in the one of~\eqref{eqt:num_literature_poisson}.
Moreover, the residual image of~\eqref{eqt:num_poisson} contains less texture than the one of~\eqref{eqt:num_literature_poisson}.

\begin{figure}[htbp]
    \centering
    \begin{subfigure}{0.49\textwidth}
        \centering \includegraphics[width=0.7\textwidth]{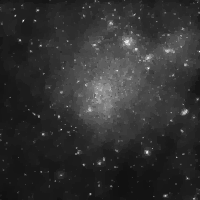}
        \caption{The restored image using~\eqref{eqt:num_poisson}}
    \end{subfigure}
    \begin{subfigure}{0.49\textwidth}
        \centering \includegraphics[width=0.7\textwidth]{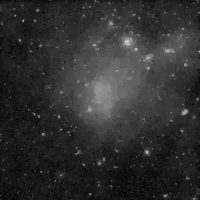}
        \caption{The restored image using~\eqref{eqt:num_literature_poisson}}
    \end{subfigure}\\
    \begin{subfigure}{0.49\textwidth}
        \centering \includegraphics[width=0.7\textwidth]{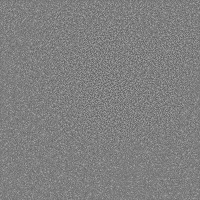}
        \caption{The residual image (+0.5) using~\eqref{eqt:num_poisson}}
    \end{subfigure}
    \begin{subfigure}{0.49\textwidth}
        \centering \includegraphics[width=0.7\textwidth]{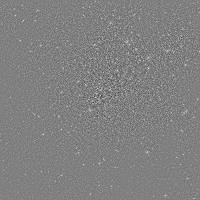}
        \caption{The residual image (+0.5) using~\eqref{eqt:num_literature_poisson}}
    \end{subfigure}
    \caption{The four figures show the restored and residual images of the model~\eqref{eqt:num_poisson} with $\alpha=2.0$ and the model~\eqref{eqt:num_literature_poisson} with $\alpha=8.21875$. 
    The parameters $\alpha$ are chosen such that the $\ell^2$-norms of the residual images of the two models are close to each other (both are around $9.9$).
\label{fig:poisson_eg3_images}}
\end{figure}

In the third example, the input image $\frac{x}{t}$ is the right image in Figure~\ref{fig:noisy_images_poisson}. We consider $t=60.0$,  $\alpha=2.0$ in~\eqref{eqt:num_poisson}, and $\alpha=8.21875$ in~\eqref{eqt:num_literature_poisson}.  Figure~\ref{fig:poisson_eg3_images} illustrates a similar behavior as in the second example. 
Compared with the output image of~\eqref{eqt:num_literature_poisson}, our model~\eqref{eqt:num_poisson} provides a restored image with less staircase effect, and it preserves more texture (see the corresponding residual image).

\subsection{Multiplicative noise}\label{sec:numerical_multiplicative}
In this section, we test our model~\eqref{eqt: mainresult_var1} on denoising problems with multiplicative noise. We consider the following optimization problem
\begin{equation} \label{eqt:numerical_multiplicative}
\min_{\imageq\in(0,+\infty)^n} \left\{t\sum_{i=1}^n\left(-1+\log \imageq_i +  \frac{x_i/t}{\imageq_i} -\log\frac{x_i}{t}\right)+ \alpha \tv\left(-\frac{1}{\imageq_1},\dots, -\frac{1}{\imageq_n}\right)\right\},
\end{equation}
where $t$ is the positive parameter which denotes the number of observed images (cf. Remark~\ref{rem:meaningt_multiplicative}), $\frac{x}{t}$ is the input noisy image, 
$\alpha$ is a positive parameter in the model, and $\tv$ is the anisotropic total variation. 
This model  is presented in section~\ref{sec: multiplicative}, where the function $J$ is the Legendre-Fenchel transform of $\alpha \tv$, that is the indicator function of the Meyer ball with radius $\alpha$. 
Hence, assumption (A1) is satisfied.
By Proposition~\ref{prop: mainresult_Moreau}, the model~\eqref{eqt:numerical_multiplicative} is equivalent to the following convex optimization problem
\begin{equation}\label{eqt:numerical_multiplicative_additive}
\min_{\imageq\in(0,+\infty)^n} \left\{J\left(x-t\imageq\right) - t\sum_{i=1}^n \log\imageq_i\right\} = \min_{\imageq\in(0,+\infty)^n} \left\{\tv^*\left(\frac{x-t\imageq}{\alpha}\right) - t\sum_{i=1}^n \log\imageq_i\right\}.
\end{equation}
We apply ADMM to solve~\eqref{eqt:numerical_multiplicative_additive}, where the $k$-th iteration reads
\begin{equation}\label{eqt:admm_multiplicative_additive}
\begin{split}
    \vkp &= \argmin_{\imageq\in(0,+\infty)^n} \left\{ - t\sum_{i=1}^n \log\imageq_i + \frac{\lambda}{2} \left\| \wk + tv -x+ \yk\right\|^2\right\},\\
     \wkp &= \argmin_{w\in \Rn} \left\{ J(w) + \frac{\lambda}{2}\left\|w + t\vkp -x+ \yk\right\|^2\right\} \\
     &= -t\vkp +x- \yk - \argmin_{z\in \Rn} \left\{ \alpha \tv(z) + \frac{1}{2}\left\|z + t\vkp -x+ \yk\right\|^2\right\},\\
     \ykp &= \yk + \wkp + t\vkp - x.
\end{split}
\end{equation}
Similar to the first line in~\eqref{eqt:admm_literature_poisson}, the first order optimality condition of the optimization problem in the first line of~\eqref{eqt:admm_multiplicative_additive} gives an analytical formula for $\vkp$, whose $i$-th component reads
\begin{equation*}
    \vkp_i = s_i + \sqrt{s_i^2 + \frac{1}{\lambda t}}, \quad \text{ where } \quad s_i = \frac{1}{2t}\left(-\wk_i - \yk_i +x_i\right).
\end{equation*}
The second line in~\eqref{eqt:admm_multiplicative_additive} involves the proximal map of the anisotropic total variation, which is approached  using the algorithm in~\cite{chambolle.09.ijcv,darbon2006imageI,hochbaum.01.jacm}.

We compare the output of~\eqref{eqt:numerical_multiplicative} with the one of the following widely used model in the literature~\cite{Aubert2008Variational} for  denoising problems with multiplicative noise
\begin{equation}\label{eqt:numerical_multiplicative_literature}
    \min_{\imageq\in(0,+\infty)^n} \left\{t\sum_{i=1}^n\left(-1+\log \imageq_i +  \frac{x_i/t}{\imageq_i} -\log\frac{x_i}{t}\right)+ \alpha \tv\left(\log\imageq_1,\dots, \log\imageq_n\right)\right\}.
\end{equation}
Recall that this is not covered by our model~\eqref{eqt: mainresult_var1}, cf.  Remark~\ref{rem:mul_tvlog_prior}. Note that \eqref{eqt:numerical_multiplicative_literature} is    equivalent to the following convex optimization problem
\begin{equation}\label{eqt:numerical_multiplicative_literature_equiv}
    \min_{w\in\R^n} \left\{\sum_{i=1}^n\left(tw_i +  x_i\exp(-w_i) \right)+ \alpha \tv\left(w\right)\right\},
\end{equation}
via the change of variable $w_i = \log \imageq_i$ for each $i\in\{1,\dots,n\}$.
We apply the ADMM method to approximate the minimizer $\bar{w}$ in~\eqref{eqt:numerical_multiplicative_literature_equiv}, and then obtain the minimizer $\bar{v}$ in~\eqref{eqt:numerical_multiplicative_literature} by $\bar{v} = \left(\exp(\bar{w}_1),\dots, \exp(\bar{w}_n)\right)$.
The $k$-th step in the ADMM method reads
\begin{equation}\label{eqt:admm_multiplicative_literature}
\begin{split}
    \ukp &= \argmin_{u\in \Rn} \left\{ \sum_{i=1}^n\left(tu_i +  x_i\exp(-u_i) \right) + \frac{\lambda}{2} \left\|u- \wk + \yk\right\|^2\right\},\\
     \wkp &= \argmin_{w\in \Rn} \left\{ \alpha \tv(w) + \frac{\lambda}{2}\left\|\ukp- w + \yk\right\|^2\right\}, \\
     \ykp &= \yk + \ukp - \wkp,
\end{split}
\end{equation}
where the first line is solved by Newton's method, and the second line is performed using the algorithm in~\cite{chambolle.09.ijcv,darbon2006imageI,hochbaum.01.jacm}.

\begin{figure}[htbp]
    \centering
    \begin{subfigure}{0.32\textwidth}
        \centering \includegraphics[width=0.9\textwidth]{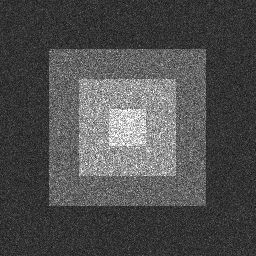}
        \caption{}
    \end{subfigure}
    \begin{subfigure}{0.32\textwidth}
        \centering \includegraphics[width=0.88\textwidth]{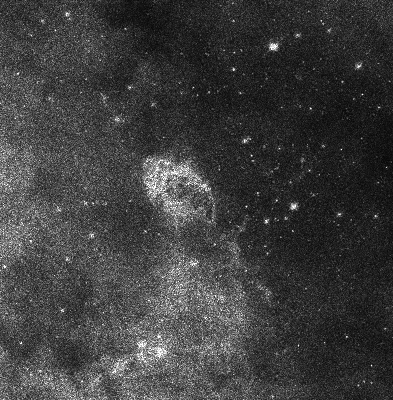}
        \caption{}
    \end{subfigure}
    \begin{subfigure}{0.32\textwidth}
        \centering \includegraphics[width=0.9\textwidth]{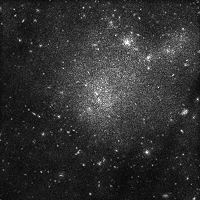}
        \caption{}
    \end{subfigure}
    \caption{The noisy images in the three testing examples corrupted by multiplicative noise.}
    \label{fig:noisy_images_multiplicative}
\end{figure}

We still use the three images in Figure~\ref{fig:ori_images} to compare the output of the two models~\eqref{eqt:numerical_multiplicative} and~\eqref{eqt:numerical_multiplicative_literature}. The corresponding noisy images shown in Figure~\ref{fig:noisy_images_multiplicative} are generated with different parameters $t>0$, as described in Remark~\ref{rem:meaningt_multiplicative}. In each example, we plot both the restored images $\bar{v}$ and the residual images $\frac{x}{t}-\bar{v}$ of the two models.
Similarly as in section~\ref{sec:numerical_poisson}, to compare the numerical results of the two models, we choose different parameters $\alpha$ in different models such that the residual images have similar $\ell^2$-norms.

\begin{figure}[htbp]
    \centering
    \begin{subfigure}{0.49\textwidth}
        \centering \includegraphics[width=0.7\textwidth]{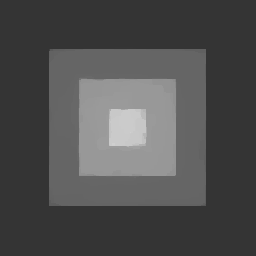}
        \caption{The restored image using~\eqref{eqt:numerical_multiplicative}}
    \end{subfigure}
    \begin{subfigure}{0.49\textwidth}
        \centering \includegraphics[width=0.7\textwidth]{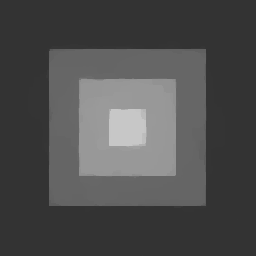}
        \caption{The restored image using~\eqref{eqt:numerical_multiplicative_literature}}
    \end{subfigure}\\
    \begin{subfigure}{0.49\textwidth}
        \centering \includegraphics[width=0.7\textwidth]{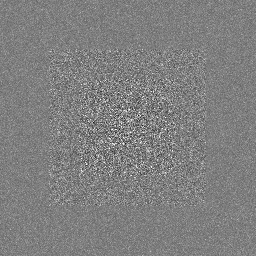}
        \caption{The residual image (+0.5) using~\eqref{eqt:numerical_multiplicative}}
    \end{subfigure}
    \begin{subfigure}{0.49\textwidth}
        \centering \includegraphics[width=0.7\textwidth]{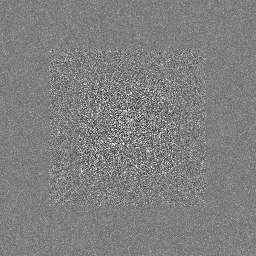}
        \caption{The residual image (+0.5) using~\eqref{eqt:numerical_multiplicative_literature}}
    \end{subfigure}\\
    \caption{The four figures show the restored and residual images of the model~\eqref{eqt:numerical_multiplicative} with $\alpha=4.0$, $t=20$ and the model~\eqref{eqt:numerical_multiplicative_literature} with $\alpha=7.0$, $t=20$. 
    The parameters $\alpha$ are chosen such that the $\ell^2$-norms of the residual images of the two models are close to each other (both are around $19.4$).
\label{fig:multiplicative_eg1}}
\end{figure}

\begin{figure}[htbp]
    \centering
    \begin{subfigure}{0.49\textwidth}
        \centering \includegraphics[width=0.7\textwidth]{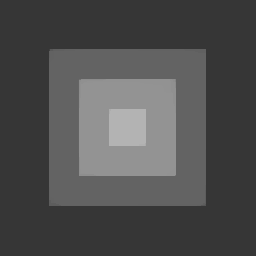}
        \caption{The restored image using~\eqref{eqt:numerical_multiplicative}}
    \end{subfigure}
    \begin{subfigure}{0.49\textwidth}
        \centering \includegraphics[width=0.7\textwidth]{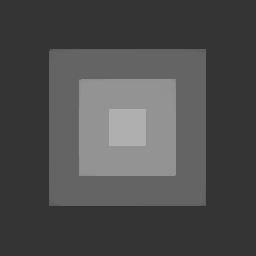}
        \caption{The restored image using~\eqref{eqt:numerical_multiplicative_literature}}
    \end{subfigure}\\
    \begin{subfigure}{0.49\textwidth}
        \centering \includegraphics[width=0.7\textwidth]{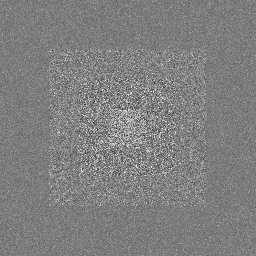}
        \caption{The residual image (+0.5) using~\eqref{eqt:numerical_multiplicative}}
    \end{subfigure}
    \begin{subfigure}{0.49\textwidth}
        \centering \includegraphics[width=0.7\textwidth]{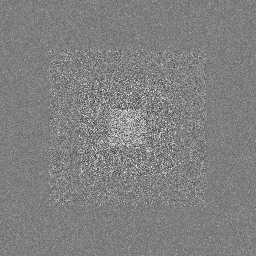}
        \caption{The residual image (+0.5) using~\eqref{eqt:numerical_multiplicative_literature}}
    \end{subfigure}
    \caption{The four figures show the restored and residual images of the model~\eqref{eqt:numerical_multiplicative} with $\alpha = 4.0$, $t=5$ and~\eqref{eqt:numerical_multiplicative_literature} with $\alpha= 7.0$, $t=5$.
\label{fig:multiplicative_eg1_images_t5}}
\end{figure}

\begin{figure}[htbp]
    \centering
    \begin{subfigure}{0.49\textwidth}
        \centering \includegraphics[width=0.7\textwidth]{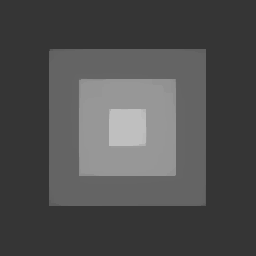}
        \caption{The restored image using~\eqref{eqt:numerical_multiplicative}}
    \end{subfigure}
    \begin{subfigure}{0.49\textwidth}
        \centering \includegraphics[width=0.7\textwidth]{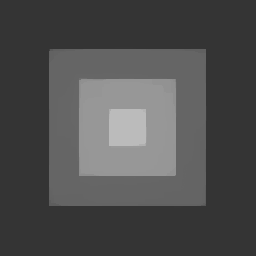}
        \caption{The restored image using~\eqref{eqt:numerical_multiplicative_literature}}
    \end{subfigure}\\
    \begin{subfigure}{0.49\textwidth}
        \centering \includegraphics[width=0.7\textwidth]{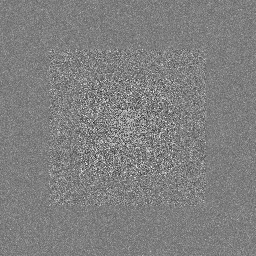}
        \caption{The residual image (+0.5) using~\eqref{eqt:numerical_multiplicative}}
    \end{subfigure}
    \begin{subfigure}{0.49\textwidth}
        \centering \includegraphics[width=0.7\textwidth]{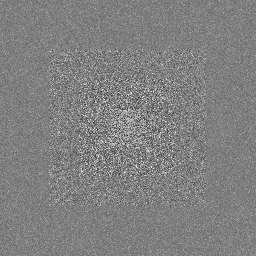}
        \caption{The residual image (+0.5) using~\eqref{eqt:numerical_multiplicative_literature}}
    \end{subfigure}
    \caption{The four figures show the restored and residual images of the model~\eqref{eqt:numerical_multiplicative} with $\alpha = 4.0$, $t=10$ and~\eqref{eqt:numerical_multiplicative_literature} with $\alpha= 7.0$, $t=10$.
\label{fig:multiplicative_eg1_images_t10}}
\end{figure}

\begin{figure}[htbp]
    \centering
    \begin{subfigure}{0.49\textwidth}
        \centering \includegraphics[width=0.7\textwidth]{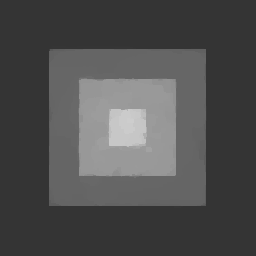}
        \caption{The restored image using~\eqref{eqt:numerical_multiplicative}}
    \end{subfigure}
    \begin{subfigure}{0.49\textwidth}
        \centering \includegraphics[width=0.7\textwidth]{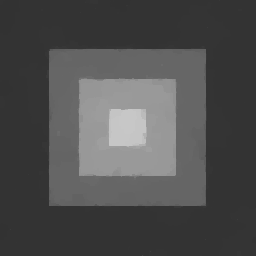}
        \caption{The restored image using~\eqref{eqt:numerical_multiplicative_literature}}
    \end{subfigure}\\
    \begin{subfigure}{0.49\textwidth}
        \centering \includegraphics[width=0.7\textwidth]{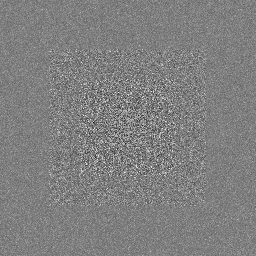}
        \caption{The residual image (+0.5) using~\eqref{eqt:numerical_multiplicative}}
    \end{subfigure}
    \begin{subfigure}{0.49\textwidth}
        \centering \includegraphics[width=0.7\textwidth]{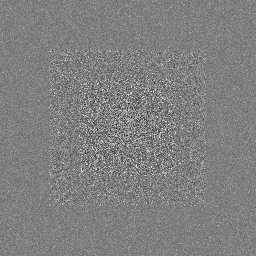}
        \caption{The residual image (+0.5) using~\eqref{eqt:numerical_multiplicative_literature}}
    \end{subfigure}
    \caption{The four figures show the restored and residual images of the model~\eqref{eqt:numerical_multiplicative} with $\alpha = 4.0$, $t=30$ and~\eqref{eqt:numerical_multiplicative_literature} with $\alpha= 7.0$, $t=30$.
\label{fig:multiplicative_eg1_images_t30}}
\end{figure}

We first apply the two variational models~\eqref{eqt:numerical_multiplicative} and~\eqref{eqt:numerical_multiplicative_literature} to the left image in Figure~\ref{fig:noisy_images_multiplicative}, which is generated with  $t = 20$. 
The parameter $\alpha$ in~\eqref{eqt:numerical_multiplicative} is taken as $4.0$, while  $\alpha$ in~\eqref{eqt:numerical_multiplicative_literature} is set to be $7.0$.
The corresponding restored images and residual images of the two models are shown in Figure~\ref{fig:multiplicative_eg1}. From the numerical results, we observe similar outputs of both variational models. These two models  perform well, in the sense of removing noise and preserving edges.
To show the influence of the parameter $t$, we fix the parameter $\alpha$ in both models and show the output images with $t=5$, $t=10$, and $t=30$ in Figures~\ref{fig:multiplicative_eg1_images_t5},~\ref{fig:multiplicative_eg1_images_t10}, and~\ref{fig:multiplicative_eg1_images_t30}, respectively. Similarly as in the Poisson case, we deduce that the parameter $t$ behaves like a regularization parameter (the smaller $t$ is, the stronger the regularization is).

\begin{figure}[htbp]
    \centering
    \begin{subfigure}{0.49\textwidth}
        \centering \includegraphics[width=0.7\textwidth]{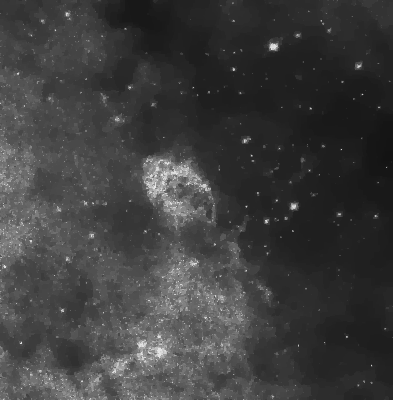}
        \caption{The restored image using~\eqref{eqt:numerical_multiplicative}}
    \end{subfigure}
    \begin{subfigure}{0.49\textwidth}
        \centering \includegraphics[width=0.7\textwidth]{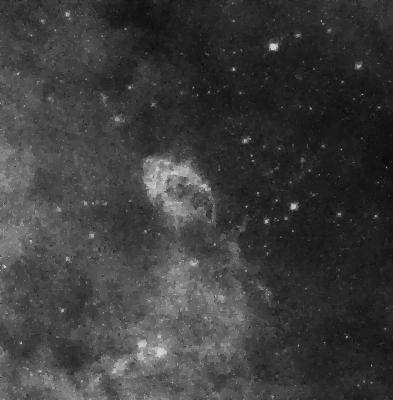}
        \caption{The restored image using~\eqref{eqt:numerical_multiplicative_literature}}
    \end{subfigure}\\
    \begin{subfigure}{0.49\textwidth}
        \centering \includegraphics[width=0.7\textwidth]{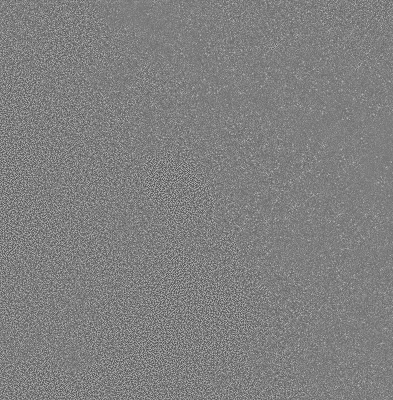}
        \caption{The residual image (+0.5) using~\eqref{eqt:numerical_multiplicative}}
    \end{subfigure}
    \begin{subfigure}{0.49\textwidth}
        \centering \includegraphics[width=0.7\textwidth]{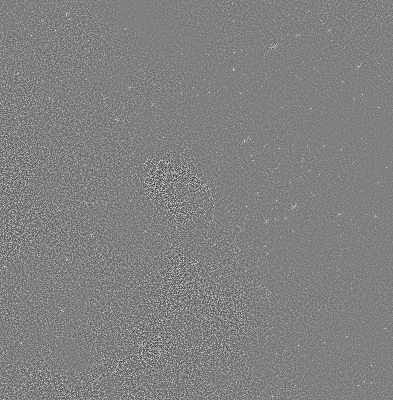}
        \caption{The residual image (+0.5) using~\eqref{eqt:numerical_multiplicative_literature}}
    \end{subfigure}\\
    \caption{The four figures show the restored and residual images of the model~\eqref{eqt:numerical_multiplicative} with $\alpha=0.5$, $t=10$ and the model~\eqref{eqt:numerical_multiplicative_literature} with $\alpha\approx 1.40$, $t=10$. 
    The parameters $\alpha$ are chosen such that the $\ell^2$-norms of the residual images of the two models are close to each other (both are around $27.5$).
\label{fig:multiplicative_eg2_alp_0p5}}
\end{figure}

\begin{figure}[htbp]
    \centering
    \begin{subfigure}{0.49\textwidth}
        \centering \includegraphics[width=0.7\textwidth]{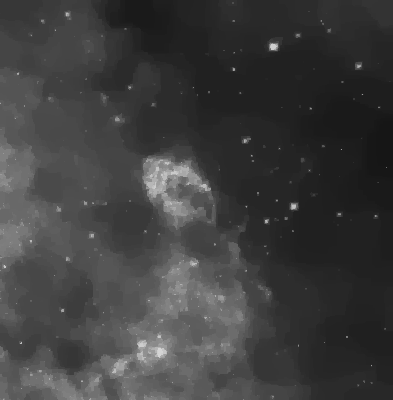}
        \caption{The restored image using~\eqref{eqt:numerical_multiplicative}}
    \end{subfigure}
    \begin{subfigure}{0.49\textwidth}
        \centering \includegraphics[width=0.7\textwidth]{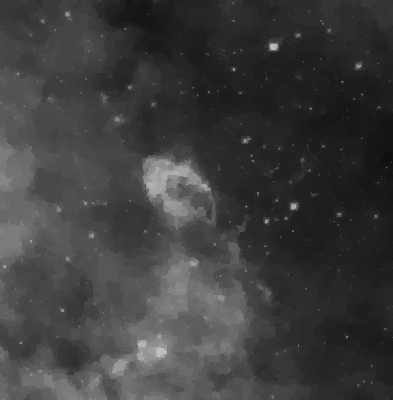}
        \caption{The restored image using~\eqref{eqt:numerical_multiplicative_literature}}
    \end{subfigure}\\
    \begin{subfigure}{0.49\textwidth}
        \centering \includegraphics[width=0.7\textwidth]{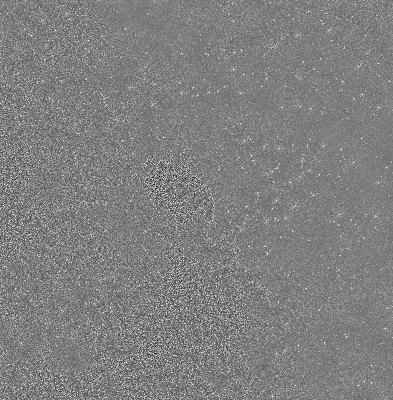}
        \caption{The residual image (+0.5) using~\eqref{eqt:numerical_multiplicative}}
    \end{subfigure}
    \begin{subfigure}{0.49\textwidth}
        \centering \includegraphics[width=0.7\textwidth]{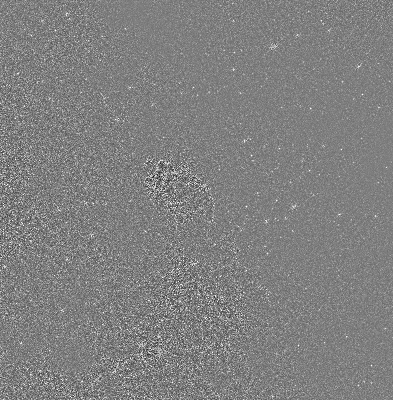}
        \caption{The residual image (+0.5) using~\eqref{eqt:numerical_multiplicative_literature}}
    \end{subfigure}\\
    \caption{The four figures show the restored and residual images of the model~\eqref{eqt:numerical_multiplicative} with $\alpha=1.0$, $t=10$ and the model~\eqref{eqt:numerical_multiplicative_literature} with $\alpha\approx 2.61$, $t=10$. 
    The parameters $\alpha$ are chosen such that the $\ell^2$-norms of the residual images of the two models are close to each other (both are around $33.1$).
\label{fig:multiplicative_eg2_alp1}}
\end{figure}

In the second example, the input is the middle image of Figure~\ref{fig:noisy_images_multiplicative}. It is generated with the parameter $t = 10$. In Figure~\ref{fig:multiplicative_eg2_alp_0p5}, we show the restored images and residual images of the model~\eqref{eqt:numerical_multiplicative} with $\alpha = 0.5$, and of the model~\eqref{eqt:numerical_multiplicative_literature} with $\alpha$ around $1.40$. The outputs of~\eqref{eqt:numerical_multiplicative} with $\alpha = 1.0$ and~\eqref{eqt:numerical_multiplicative_literature} with $\alpha$ around $2.61$ are illustrated in Figure~\ref{fig:multiplicative_eg2_alp1}. When the parameter $\alpha$ increases,  more regularized restored images are obtained by both models. 
We observe less staircase in the restored image of our model~\eqref{eqt:numerical_multiplicative} than in the one of~\eqref{eqt:numerical_multiplicative_literature}, as well as less texture in the residual image of our model~\eqref{eqt:numerical_multiplicative} as compared to the one of~\eqref{eqt:numerical_multiplicative_literature}.

\begin{figure}[htbp]
    \centering
    \begin{subfigure}{0.49\textwidth}
        \centering \includegraphics[width=0.7\textwidth]{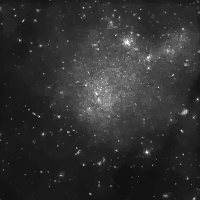}
        \caption{The restored image using~\eqref{eqt:numerical_multiplicative}}
    \end{subfigure}
    \begin{subfigure}{0.49\textwidth}
        \centering \includegraphics[width=0.7\textwidth]{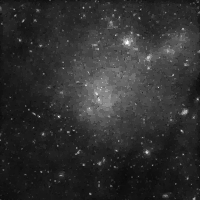}
        \caption{The restored image using~\eqref{eqt:numerical_multiplicative_literature}}
    \end{subfigure}\\
    \begin{subfigure}{0.49\textwidth}
        \centering \includegraphics[width=0.7\textwidth]{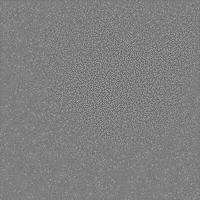}
        \caption{The residual image (+0.5) using~\eqref{eqt:numerical_multiplicative}}
    \end{subfigure}
    \begin{subfigure}{0.49\textwidth}
        \centering \includegraphics[width=0.7\textwidth]{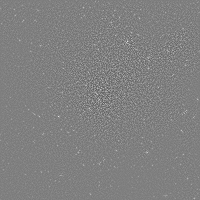}
        \caption{The residual image (+0.5) using~\eqref{eqt:numerical_multiplicative_literature}}
    \end{subfigure}\\
    \caption{The four figures show the restored and residual images of the model~\eqref{eqt:numerical_multiplicative} with $\alpha=0.5$, $t=20$ and the model~\eqref{eqt:numerical_multiplicative_literature} with $\alpha\approx 1.68$, $t=20$. 
    The parameters $\alpha$ are chosen such that the $\ell^2$-norms of the residual images of the two models are close to each other (both are around $7.4$).
\label{fig:multiplicative_eg3}}
\end{figure}

In the third example, the input noisy image $\frac{x}{t}$ is the right image in Figure~\ref{fig:noisy_images_multiplicative}, which is generated with the parameter $t = 20$. 
We set the parameter $\alpha$ in the model~\eqref{eqt:numerical_multiplicative} to be $0.5$, and the parameter $\alpha$ in the model~\eqref{eqt:numerical_multiplicative_literature} to be around $1.68$. The restored images and residual images of the two models are shown in Figure~\ref{fig:multiplicative_eg3}.
The restored image of~\eqref{eqt:numerical_multiplicative_literature} contains less noise, while the residual image of~\eqref{eqt:numerical_multiplicative} contains less texture (similarly to the second example). 
Compared to the widely used model~\eqref{eqt:numerical_multiplicative_literature}, our model~\eqref{eqt:numerical_multiplicative} provides restored images which have less staircase effect and meanwhile preserve more edges or texture in the original images.

\section{Summary}
In this paper, we propose the variational model~\eqref{eqt: mainresult_var1} for denoising problems with non-additive noise and provide its connection with the additive denoising model~\eqref{eqt: mainresult_var2} and certain HJ PDEs. In this way, we point out new possibilities to apply techniques and algorithms from one area to solve problems in the other area. As a by-product, our results show that the solution of certain initial-valued HJ problems in many space dimensions can be computed at some points using optimization solvers already developed for imaging purposes.
We also hope that a future better understanding of these HJ PDEs will help creating more efficient numerical optimization algorithms for the variational models considered in this paper. Moreover, promising numerical results are pointed out for the non-convex model~\eqref{eqt: mainresult_var1} arising in denoising problems involving Poisson noise or multiplicative noise, based on  robust algorithms for solving the associated convex optimization problem~\eqref{eqt: mainresult_var2}.
Note that the ADMM algorithm~\eqref{eqt:numerical_admm} in section~\ref{sec:numerics} works for any model~\eqref{eqt: mainresult_var1} satisfying assumption (A1).

Note that in this paper we consider non-additive noise models whose data fidelity are in the form of Bregman distance $tD_{H^*}\left(\frac{x}{t}, \imageq\right)$ for some function $H$. This framework covers many widely-used denoising problems, including the model where the observed image is assumed to be sampled from certain exponential family $P(\frac{x}{t}|t,v)$ (e.g., Poisson noise and multiplicative noise). 
A future direction is to study the connections between HJ PDEs and the linear inverse problems where the observed data is sampled from $P(\frac{x}{t}|t,Av)$ for some non-invertible linear operator $A$ (e.g., image deblurring problems).
Also, this paper only considers the noise which comes from a single distribution. Another challenge is to investigate the connections of HJ PDEs and the variational problems with mixed noise, such as the mixture of Gaussian and Poisson noise.

\begin{acknowledgements}
The authors gratefully acknowledge the support of IPAM (UCLA), where this collaboration started during the ``High Dimensional Hamilton-Jacobi PDEs" program (2020).
This work was funded by NSF DMS 1820821.
The authors thank the reviewers for providing fruitful comments that improved the presentation of the paper.
\end{acknowledgements}

\bibliographystyle{spmpsci}      
\bibliography{references}

\newpage
\appendix
\section{Mathematical background} \label{appendix:bkgd}
In this section, several basic definitions and theorems in convex analysis are reviewed. All the results and notations can be found in \cite{ekeland1976,convex_analy_book1,convex_analy_book2,hiriart2004}. 
We use the angle bracket $\langle \cdot, \cdot \rangle$ to denote the inner product operator in any Euclidean space $\R^n$.

We denote $\gmRn$ to be the set of proper, convex and lower semi-continuous functions from $\Rn$ to $\R\cup\{+\infty\}$. Recall that a function $f\colon\R^n\to\R\cup\{\pm\infty\}$ is called proper if there exists a point $x$ such that $f(x)\in\R$.
In this paper, we assume every function is proper if not mentioned specifically.
The functions in $\gmRn$ have good continuity properties, which are stated below.

\begin{proposition} \cite[Lem.IV.3.1.1 and Chap.I.3.1 - 3.2]{convex_analy_book1}
\label{prop:bkgd_convex_cont} 
Let $f\in \gmRn$. If $x\in \ri \dom f$, then $f$ is continuous at $x$ in $\dom f$. If $x\in \dom f \setminus \ri \dom f$, then for all $y\in \ri \dom f$,
\begin{equation*}
f(x) = \lim_{t\to 0^+} f(x + t(y-x)).
\end{equation*}
\end{proposition}

A vector $p$ is called a subgradient of $f$ at $x$ if it satisfies
\begin{equation*}
f(y)\geq f(x) + \langle p, y-x\rangle, \quad \forall y\in\Rn.
\end{equation*}
The collection of all such subgradients is called the subdifferential of $f$ at $x$, denoted as $\partial f(x)$. It is straightforward to check that $0\in \partial f(x)$ if and only if $x$ is a minimizer of $f$. As a result, one can check whether $x$ is a minimizer by computing the subdifferential.

Moreover, in most cases, the subdifferential operator commutes with summation. One set of assumption is given in the following proposition.
\begin{proposition}
\cite[Proposition~5.6, Ch.~I]{ekeland1976}
\label{prop: bkgd_linear_subdiff}
Let $f,g\in \gmRn$. Assume there exists a point $u\in \dom f\cap \dom g$ where $f$ is continuous.
Then we have $\partial (f+g)(x) = \partial f(x) + \partial g(x)$ for all $x\in \dom f \cap \dom g$.
\end{proposition}

One important class of functions in $\gmRn$ is called indicator functions. For any convex set $C$, the indicator function $I_C$ is defined by
\begin{equation*}
    I_C(x) := \begin{cases}
    0, & \text{if }x\in C,\\
    +\infty, &\text{if }x\not\in C.
    \end{cases}
\end{equation*}
One can compute the subdifferential of the indicator function and obtain
\begin{equation} \label{eqt:subgrad_indicator}
    \partial I_C(x) = 
    N_C(x) ,
\end{equation}
where $N_C(x)$ denotes the normal cone of $C$ at $x$. Note that $N_C(x) = \{0\}$ when $x$ is in the interior of the set $C$ (see~\cite[p. 137]{convex_analy_book1} for details).

\bigbreak
Next, we recall one important transform in convex analysis called Legendre-Fenchel transform. For all $f\in \gmRn$, the Legendre-Fenchel transform of $f$, denoted as $f^*$, is defined by
\begin{equation}\label{eqt:Legendre}
f^*(p) := \sup_{x\in \Rn} \{\langle p, x\rangle - f(x)\}.
\end{equation}
The Legendre-Fenchel transform gives a duality relationship between $f$ and $f^*$. In other words, if $f\in \gmRn$ holds, then there hold $f^*\in \gmRn$ and $f^{**} = f$.
Similarly, along with this duality, some properties are dual to others, as stated in the following proposition. Recall that a function $g$ is called 1-coercive if $\lim_{\|x\|\to +\infty} \frac{g(x)}{\|x\|} = +\infty$.

\begin{proposition} 
\cite[Propositions E.1.3.8 and E.1.3.9]{hiriart2004}
\label{prop: bkgd_finite_coercive}
Let $f\in \gmRn$. Then $f$ is finite-valued if and only if $f^*$ is 1-coercive. 
\end{proposition}

In particular, the subdifferential set of $f^*$ is characterized by the maximizers in \eqref{eqt:Legendre}, as stated in the following proposition.

\begin{proposition} \cite[Corollary E.1.4.4]{hiriart2004}
\label{prop: bkgd_conjugate_equal_inverse}
Let $f\in \gmRn$ and $p,x \in \Rn$. Then $p\in \partial f(x)$ holds if and only if $x\in \partial f^*(p)$ holds, which is equivalent to $f(x) + f^*(p) = \langle p, x\rangle$.
\end{proposition}

Except from the Legendre-Fenchel transform, there is another operator defined on convex functions, called inf-convolution. Given two functions $f,g\in \gmRn$, assume there exists an affine function $l$ such that $f(x)\geq l(x)$ and $g(x)\geq l(x)$ for all $x\in\Rn$. Then, the inf-convolution between $f$ and $g$, denoted by $f\conv g$, is a convex function taking values in $\R\cup \{+\infty\}$ defined by
\begin{equation} \label{eqt:infconv}
(f\conv g)(x) := \inf_{u\in \Rn} \{f(u) + g(x-u)\}.
\end{equation}
In the following proposition, the relation between Legendre-Fenchel transform and inf-convolution is stated. 

\begin{proposition} 
\cite[Theorem~E.2.3.2]{hiriart2004}
\label{prop: bkgd_sum_Legendre_infconv}
Let $f,g \in \gmRn$. Assume the intersection of $\ri \dom f$ and $\ri \dom g$ is non-empty. Then there hold $f^*\conv g^*\in \gmRn$ and $f^*\conv g^* = (f+g)^*$. 
Moreover, the minimization problem
\begin{equation}
    \min_{u\in \Rn} \{f^*(u) + g^*(x-u)\}
\end{equation}
has at least one minimizer whenever $x$ is in $\dom f^*\conv g^*=\dom f^* + \dom g^*$.
\end{proposition}

Let $f$ be a function in $\gmRn$ and $x_0$ be an arbitrary point in $\dom f$. The asymptotic function of $f$, denoted by $f'_\infty$, is defined by 
\begin{equation}\label{eqt:defasym}
    f'_\infty(d) = \sup_{s>0} \frac{f(x_0 + sd) - f(x_0)}{s} = \lim_{s\to+\infty} \frac{f(x_0 + sd) - f(x_0)}{s},
\end{equation}
for all $d\in \R^n$. In fact, this definition does not depend on the point $x_0$.

\begin{proposition} \cite[Example B.3.2.3]{hiriart2004}
\label{prop: bkgd_asym1}
Let $f\in\gmRn$. Take $x_0\in\dom f$. Then the function defined by 
\begin{equation}
    \R^n\times\R \ni (d,t)\mapsto 
    \begin{dcases}
    t\left(f\left(x_0+\frac{d}{t}\right) - f(x_0)\right), & \text{if }t>0,\\
    f'_\infty(d), &\text{if } t=0,\\
    +\infty, &\text{if }t<0,
    \end{dcases}
\end{equation}
is in $\Gamma_0(\R^n\times\R)$. As a result, we have
\begin{equation}
    \liminf_{t\to 0^+, d\to d_0} t\left(f\left(x_0+\frac{d}{t}\right) - f(x_0)\right)
    \geq f'_\infty(d_0).
\end{equation}
\end{proposition}

\begin{proposition} \cite[Proposition~E.1.2.2]{hiriart2004}
\label{prop: bkgd_asym2}
Let $f\in\gmRn$. Then the support function of $\dom f$ is given by the asymptotic function of $f^*$ as follows
\begin{equation}
    I_{\dom f}^* = (f^*)'_\infty.
\end{equation}
\end{proposition}

There is another important class of functions called Legendre functions, whose definition is given as follows. Note that here we use an equivalent definition to simplify our proofs.

\begin{definition}
\cite[Section~26]{Rockafellar1970Convex} Suppose $f$ is a function in $\gmRn$. Then $f$ is Legendre (or a Legendre function or a convex function of Legendre type), if $f$ satisfies
\begin{enumerate}
    \item $\interior \dom f$ is non-empty.
    \item $f$ is differentiable on $\interior \dom f$.
    \item The subdifferential of $f$ satisfies
    \begin{equation}
        \partial f(x) = \begin{dcases}
        \emptyset, & \text{if }x\in (\dom f)\setminus (\interior \dom f),\\
        \{\nabla f(x)\}, & \text{if }x\in\interior \dom f.
        \end{dcases}
    \end{equation}
    \item $f$ is strictly convex on $\interior \dom f$.
\end{enumerate}
\end{definition}

For a Legendre function $f$, we have the following property which is used in proofs of this paper. For more properties of Legendre functions, we refer readers to~\cite[Section~26]{Rockafellar1970Convex} and~\cite{Bauschke1997Legendre}.

\begin{proposition}\cite[Theorem~26.5]{Rockafellar1970Convex}\label{prop: bkgd_Legendre}
Let $f$ be a function in $\gmRn$.
Then, $f$ is a Legendre function
if and only if $f^*$ is a Legendre function.
\end{proposition}

\bigskip

Now, we are going to define the primal-dual Bregman distance denoted by $d$ and the (primal-primal) Bregman distance denoted by $D$. For properties of Bregman distances, we refer readers to~\cite{Bauschke1997Legendre} and \cite{BREGMAN1967200}, for instance.
For $\fBreg\in \gmRn$, the primal-dual Bregman distance with respect to $\fBreg$ between a vector $x$ in the primal space $\R^n$ and a vector $p$ in the dual space $\R^n$, denoted by $d_\fBreg(x,p)$, is defined as follows
\begin{equation} \label{eqt: def_Bregd}
    d_\fBreg(x,p) := \fBreg(x) + \fBreg^*(p) - \langle p,x\rangle.
\end{equation}
When we have $\fBreg\in \gmRn$, since the equality $(\fBreg^*)^* = \fBreg$ holds, we have
\begin{equation}\label{eqt: equal_Bregd}
    d_{\fBreg^*}(p,x) = \fBreg^*(p) + (\fBreg^*)^*(x) - \langle x,p\rangle = d_{\fBreg}(x,p).
\end{equation}

Additionally, for two points $x,u\in \dom \fBreg$, let $\fBreg$ be  differentiable at $u$.
The Bregman distance $D_\fBreg$ between $x$ and $u$ with respect to the function $\fBreg$ is defined as 
\begin{equation}\label{eqt: def_BregD}
D_\fBreg(x,u)= \fBreg(x)- \fBreg(u) - \langle \nabla \fBreg(u),x-u\rangle.    
\end{equation}

These two definitions are related to each other by the following result.

\begin{proposition} \label{prop: Breg_equal_twodef} 
Let $\fBreg\in \gmRn$.
Then for all $x,u\in \dom \fBreg$ such that $\fBreg$ is differentiable at $u$, we have
\begin{equation}
    d_\fBreg(x,\nabla \fBreg(u)) = D_\fBreg(x,u).
\end{equation}
\end{proposition}

\begin{proof}
For all $x,u\in \dom \fBreg$, we have
\begin{equation}
\begin{split}
    d_\fBreg(x,\nabla \fBreg(u)) &= \fBreg(x) + \fBreg^*(\nabla \fBreg(u)) - \langle \nabla \fBreg(u),x\rangle\\
    &= \fBreg(x) + \langle \nabla \fBreg(u), u\rangle - \fBreg(u) - \langle \nabla \fBreg(u), x\rangle \\
    &= D_\fBreg(x,u),
\end{split}
\end{equation}
where the second equality holds by Proposition~\ref{prop: bkgd_conjugate_equal_inverse}.
\end{proof}

We also have the following equality for $D_\fBreg$.

\begin{proposition}\label{prop: Breg_scaling}
Let $\fBreg\in \gmRn$. 
Let $t>0$ and $\frac{x}{t}, \frac{u}{t}\in\dom \fBreg^*$. Assume $f^*$ is differentiable at 
$\frac{u}{t}$.
Then, we have
\begin{equation}
    tD_{\fBreg^*}\left(\frac{x}{t}, \frac{u}{t}\right)
    = D_{(t\fBreg)^*}(x,u).
\end{equation}
\end{proposition}

\begin{proof}
By definition of $D_{\fBreg^*}$ in~\eqref{eqt: def_BregD}, it follows that
\begin{equation}\label{eqt: Breg_equality_pf1}
\begin{split}
    tD_{\fBreg^*}\left(\frac{x}{t}, \frac{u}{t}\right)
    &= t\left(\fBreg^*\left(\frac{x}{t}\right) -\fBreg^*\left(\frac{u}{t}\right) - \left\langle \nabla \fBreg^*\left(\frac{u}{t}\right), \frac{x-u}{t}\right\rangle \right)\\
    &= t\fBreg^*\left(\frac{x}{t}\right) -t\fBreg^*\left(\frac{u}{t}\right) - \left\langle \nabla \fBreg^*\left(\frac{u}{t}\right), x-u\right\rangle.
\end{split}
\end{equation}
Define $\gBreg\colon \R^n\to\R\cup\{+\infty\}$ by $\gBreg(w):= t\fBreg^*(\frac{w}{t})$ for all $w\in \R^n$, then there hold $\nabla \gBreg(u) = \nabla \fBreg^*(\frac{u}{t})$ and $\gBreg = (t\fBreg)^*$.
And hence by~\eqref{eqt: Breg_equality_pf1} we obtain
\begin{equation}
    tD_{\fBreg^*}\left(\frac{x}{t}, \frac{u}{t}\right)
    = \gBreg(x) - \gBreg(u) - \langle \nabla \gBreg(u), x-u\rangle
    = D_{\gBreg}(x,u) = D_{(t\fBreg)^*}(x,u),
\end{equation}
which proves the statement.
\end{proof}

\end{document}